\documentclass[12pt,onecolumn]{article}
\usepackage[mathscr]{euscript}
\usepackage{mismath}
\usepackage{pst-node}
\usepackage{tikz-cd} 
\usepackage{tikz} 
\usepackage{pgfplots}
\usepgfplotslibrary{fillbetween}
\usetikzlibrary{intersections, pgfplots.fillbetween}
\usepackage{float}
\usetikzlibrary {arrows.meta}
\usetikzlibrary{decorations.markings}
\usepackage{xfrac}
\usepackage{bm}
\usepackage{graphicx}
\usepackage[final]{pdfpages} 
\usepackage{enumerate} 
\usepackage{enumitem}
\setlist[enumerate]{noitemsep}
\setlist[itemize]{noitemsep}
\usepackage{amssymb}
\usepackage{amsmath} 
\usepackage{amsthm}
\usepackage{bm}
\usepackage{bibentry} 
\usepackage[T1]{fontenc}
\usepackage{pbsi}
\usepackage{fancyvrb}
\usepackage{amssymb}
\usepackage{extarrows}
\usepackage{indentfirst}

\usepackage[labelsep=period]{caption} 

\usepackage{titlesec} 
\titleformat*{\section}{\large\bfseries} 
\titleformat*{\subsection}{\normalsize\bfseries}
\usepackage{amsthm}

\theoremstyle{plain}
\newtheorem{theorem}{Theorem}[section]
\newtheorem{corollary}[theorem]{Corollary}
\newtheorem{lemma}[theorem]{Lemma}
\newtheorem{proposition}[theorem]{Proposition}

\theoremstyle{remark}
\newtheorem{remark}[theorem]{Remark}
\newtheorem{question}[theorem]{Question}

\theoremstyle{definition}
\newtheorem{definition}[theorem]{Definition}

\newcounter{alpha}

\theoremstyle{plain}
\newtheorem{theoremalpha}[alpha]{Theorem}
\newtheorem{propositionalpha}[alpha]{Proposition}

\DeclareMathOperator{\bord}{Bord}

\DeclareMathOperator{\Top}{Top}
\DeclareMathOperator{\Sm}{Sm}

\DeclareMathOperator{\FS}{FS}

\DeclareMathOperator{\TOP}{TOP}
\DeclareMathOperator{\oo}{O}
\DeclareMathOperator{\bo}{BO}

\DeclareMathOperator{\BTOP}{BTOP}
\DeclareMathOperator{\BSpin}{BSpin}
\DeclareMathOperator{\BTopSpin}{BSPIN
TOP}

\DeclareMathOperator{\BDIFF}{BDIFF}
\DeclareMathOperator{\gamtop}{\mathfrak{g}_{TOP}}

\DeclareMathOperator{\gamo}{\gamma_{O}}

\newcommand{\exx}{{\scriptstyle{\mathfrak{X}}}}

\newcommand{\topcat}[1]{\bord^{\Top}_{{#1}}}

\newcommand{\fscat}[1]{\bord^{\FS}_{{#1}}}
\newcommand{\smcat}[1]{\bord^{\Sm}_{{#1}}}

\newcommand{\TOPn}[1]{\TOP{{(#1)}}}

\newcommand{\On}[1]{\oo{{(#1)}}}
\newcommand{\bTOPn}[1]{\BTOP{{(#1)}}}

\newcommand{\bon}[1]{\bo{{(#1)}}}
\newcommand{\gamtopn}[1]{\mathfrak{g}^{{#1}}_{\TOP}}

\newcommand{\gamon}[1]{\gamma^{{#1}}_{\oo}}

\newcommand{\micro}[1]{\mathfrak{t}_{{#1}}}

\DeclareMathOperator{\ks}{ks}
\DeclareMathOperator{\kk}{K}
\DeclareMathOperator{\vectc}{Vect_{\mathbb{C}}}
\DeclareMathOperator{\svectc}{sVect_{\mathbb{C}}}
\DeclareMathOperator{\hilb}{Hilb_{\mathbb{C}}}

\DeclareMathOperator{\sq}{Sq}
\newcommand{\sqi}[1]{\sq^{{#1}}}

\newcommand{\longrightsquigarrow}{\mathrel{\tikz{\draw[->,thick] (0,0) -- (0.05,0.05) -- (0.15,-0.05)--(0.25,0.05)-- (0.3,0) -- (0.5,0);
}}}
\newcommand{\extralongrightsquigarrow}{\mathrel{\tikz{\draw[->,thick] (0,0) -- (0.05,0.05) -- (0.15,-0.05)--(0.25,0.05)-- (0.35,-0.05)--(0.45,0.05)--(0.55,-0.05) -- (0.65,0.05) -- (0.75,-0.05) -- (0.8,0) -- (1,0);
}}}

\newcommand{\LL}{\left(}
\newcommand{\RR}{\right)}
\newcommand{\meight}{M^4_{E_8 \oplus E_8}}
\newcommand{\msixt}{M^4_{E_{16}}}

\newcommand{\ssm}[1]{{{#1}}\times S^m}

\usepackage[backend=biber,maxalphanames=99,maxbibnames=99,style=alphabetic]{biblatex}
\DeclareMathOperator{\Int}{Int}

\pgfplotsset{compat=1.18}

\usepackage{setspace}
\AtBeginDocument{\singlespacing}

\makeatletter
  \g@addto@macro\th@plain{\singlespacing}
  \g@addto@macro\th@definition{\singlespacing}
  \g@addto@macro\th@remark{\singlespacing}
  \makeatother

\usepackage{hyperref}
\hypersetup{
    colorlinks=false,
    pdfborder={0 0 0}
}
\begin{document}

\title{\large\textbf{Failure of Ordinary TQFTs to Detect Homotopy Type}}
\author{{Katherine Novey}
\date{} }
\maketitle

\begin{abstract}
   Work by David Reutter and Christopher Schommer-Pries has shown that ordinary smooth TQFTs can distinguish (stable) diffeomorphism classes of closed, connected, even-dimensional manifolds subject to certain finiteness conditions. In particular, simply connected, closed, smooth $6$-manifolds with finite $\pi_2$ are diffeomorphic if and only if they cannot be distinguished by ordinary TQFTs. The question of whether or not  this result holds for all simply connected closed $6$-manifolds was open. We consider TQFTs out of the topological and smooth bordism categories, as well as out of a new bordism category constructed in this paper called the formally smooth (FS) bordism category. We present a pair of simply connected closed $6$-manifolds with infinite $\pi_2$ that are not homotopy equivalent, yet are indistinguishable by ordinary TQFTs in all three categories. We extend this counter-example to show that ordinary TQFTs cannot distinguish the homotopy type of simply connected closed manifolds in all dimensions $\geq 6$ and non-simply connected closed $5$-manifolds. We show that our class of counter-example pairs is spinnable and present a non-spinnable class of counter-examples by introducing a $\mathbb{C}P^2$ connected summand. Lastly, we use more highly connected analogues of our counter-examples to show that, for all $n\geq2$ and $m\geq 2n$, there is a pair of $(2n-1)$-connected closed $(4n+m)$-manifolds that are not homotopy equivalent, yet cannot be distinguished by ordinary TQFTs. 
 \end{abstract}
\newpage

\renewcommand{\baselinestretch}{1.2}    
\tableofcontents
\newpage

\section{Introduction}

We will assume all manifolds to be oriented and compact, unless otherwise stated.
In this paper, we discuss manifolds in several categories. In the introduction we will write "manifold" and "isomorphism" so as to encompass all of these categories. We use "$\cong$" to indicate isomorphism.

\begin{definition}
    An $\bm{n}$\textbf{-dimensional topological quantum field theory (TQFT)} is a symmetric monoidal functor
        \[\mathcal{Z}:\bord_n \to \vectc,\]
    
    \noindent where $\bord_n$ is the $n$-dimensional bordism category, and $\vectc$ is the category of complex vector spaces and complex linear maps.\footnote{One can also define TQFTs with target category $\svectc$, the category of complex super vector spaces and super linear maps. All of the arguments in this paper hold in that case as well.} 
\end{definition}

In this paper, we restrict our attention to \textit{ordinary} TQFTs (see Definition \ref{def: ord}). All currently known examples of TQFTs in dimension at least $4$ are ordinary.\footnote{Until quite recently, this was also true in dimension $3$. However, Riva constructed a non-ordinary $3$-dimensional TQFT using a more general target (higher) category (see \cite{lorpzo1},\cite{lorpzo2}).} TQFTs yield isomorphism invariants of closed manifolds. There has been significant work done over the past couple of decades investigating the sensitivity of these invariants and which manifolds they distinguish. Although this question is still largely open, it has been shown that, in general, ordinary TQFTs do not provide complete isomorphism invariants. For even-dimensional manifolds, there is a coarser relation than isomorphism, known as stable isomorphism (see Definition \ref{def: stab iso}). Reutter and Schommer-Pries showed in \cite{rsp} that ordinary even-dimensional TQFTs can never witness any finer relation.\footnote{The work of Reutter and Schommer-Pries considers semisimple TQFTs, which are defined with weaker conditions than ordinary TQFTs. However, since all known examples of TQFTs in relevant dimensions are ordinary, we quote their results in terms of ordinary TQFTs.} The main result of this paper expands on that work and further highlights shortcomings of TQFT invariants.

\begin{theoremalpha}
    Ordinary TQFTs cannot detect the homotopy type or the stable isomorphism type of simply-connected closed manifolds in all dimensions $\geq 6$ and non simply-connected closed $5$-manifolds.
\end{theoremalpha}

\subsection{Ordinary TQFTs and Summary of Previous Results}

We will use the notation "$M: X \rightsquigarrow Y$" to denote a bordism $M$ with source $X$ and target $Y$. Applying an $n$-dimensional TQFT $\mathcal{Z}$ to the bordisms $D^n\setminus \LL D^n\sqcup D^n\RR:S^{n-1}\sqcup S^{n-1} \rightsquigarrow S^{n-1}$ and $D^n:\varnothing\rightsquigarrow S^{n-1}$ yields an algebra structure on $\mathcal{Z}\LL S^{n-1}\RR$ that is commutative for $n\geq 2$. We will refer to these bordisms as $\mu$ and $\upsilon$, respectively, since they correspond to the multiplication and unit maps of the algebra.
\begin{center}
\begin{tikzpicture}[scale=0.8]
\draw[thick] (-6,1) ellipse (0.25cm and 0.5cm);
\draw[thick] (-6,-1) ellipse (0.25cm and 0.5cm);
\draw[thick] (-4,0) ellipse (0.25cm and 0.5cm);
\draw[thick] (-6,1.5) to[out=0, in=180] (-4,0.5);
\draw[thick] (-6,-1.5) to[out=0, in=180] (-4,-0.5);
\draw[thick] (-6,0.5) to[out=0, in=90] (-5.5,0) to[out=-90, in=0] (-6,-0.5);
\draw[thick][|->] (-3.25,0) -- (-1.25,0);
\node at (3,0) {$\Big(\mathcal{Z}\left(S^{n-1}\right)\otimes \mathcal{Z}\LL S^{n-1}\RR \to \mathcal{Z}\LL S^{n-1}\RR\Big)$};
\node at (-2.25, 0.5) {$\mathcal{Z}$};
\draw[thick] (-4,-3) ellipse (0.25cm and 0.5cm);
\draw[thick] (-4,-2.5) to[out=180, in=90] (-5,-3) to[out=-90, in=180] (-4,-3.5);
\draw[thick][|->] (-3.25,-3) -- (-1.25,-3);
\node at (-5,0) {$\mu$};
\node at (-4.5,-3) {$\upsilon$};
\node at (-2.25, -2.5) {$\mathcal{Z}$};
\node at (1,-3) {$\Big(\mathbb{C}\to \mathcal{Z}\LL S^{n-1}\RR\Big)$};
\end{tikzpicture}
\end{center} 

\noindent Reversing the direction of $\mu$ and $\upsilon$ gives a coalgebra structure. Thus, $\mathcal{Z}\LL S^{n-1}\RR$ is a Frobenius algebra.

These algebra structures may be extended to a much larger class of manifolds. We do this using a technique known as dimensional reduction, which is a method of obtaining a $q$-dimensional TQFT from an $n$-dimensional one by "reducing along" some closed $(n-q)$-manifold. Let $\mathcal{Z}:\bord_n\to\vectc$ be an $n$-dimensional TQFT, and let $K$ be a closed $(n-q)$-manifold.

\begin{definition}
    The \textbf{dimensional reduction of} $\bm{\mathcal{Z}}$ \textbf{along} $\bm{K}$ is the $q$-dimensional TQFT $\mathcal{W}$ defined as the composition below. \[\mathcal{W}:\bord_{q}\xrightarrow{-\times K} \bord_n \xrightarrow{\mathcal{Z}}\vectc\]
\end{definition}

Since $\mathcal{W}(S^{q-1})=\mathcal{Z}(S^{q-1}\times K)$, the Frobenius algebra structure on $\mathcal{W}(S^{q-1})$ immediately gives rise to one on $\mathcal{Z}(S^{q-1}\times K)$. Note that, if $\mu$ and $\upsilon$ are the $q$-dimensional bordisms representing the multiplication and unit for $\mathcal{W}(S^{q-1})$, the $n$-dimensional analogues for $\mathcal{Z}(S^{q-1}\times K)$ are simply $\mu\times K$ and $\upsilon\times K$.

\begin{definition} \label{def: ord}
    An $n$-dimensional TQFT $\mathcal{Z}$ is \textbf{ordinary} if $\mathcal{Z}(S^{q-1}\times K)$ is a semisimple algebra for any closed $\LL n-q\RR$-dimensional manifold $K$, for any $q$, $1\leq q\leq n$. 

\end{definition}

In \cite{rsp}, Reutter and Schommer-Pries proved that ordinary TQFTs yield stable isomorphism invariants.

\begin{definition} \label{def: stab iso}
    \sloppy Two $2k$-dimensional manifolds $M$ and $N$ are \textbf{stably isomorphic} if, for some $p$, $M \#^p \LL S^k\times S^k\RR \cong N \#^p \LL S^k\times S^k\RR$, where "$\#^p\LL S^k\times S^k\RR$" denotes the connected sum with $p$ copies of $\LL S^k\times S^k\RR$. 
\end{definition}

\begin{theorem}[{\cite[Theorem~A]{rsp}}\footnote{The $4$-dimensional version of this result appeared earlier in \cite{reutter} as Theorem A.}]
\label{thm: rsp A}
    Stably isomorphic even-dimensional manifolds cannot be distinguished by ordinary TQFTs. 
\end{theorem}

\noindent Furthermore, they showed that, in the smooth category, these invariants are complete for closed manifolds that satisfy the following finiteness condition.

\begin{definition}[{\cite{rsp}}]
    A $2k$-dimensional smooth manifold $M$ has $\bm{\pi}$\textbf{-finite tangential} $\bm{\LL k-1\RR}$\textbf{-type} if $\pi_i F$ is finite for all $i \leq k-1$ for all basepoints, where $F$ is the homotopy fiber of a map $M\to \bon{2k}$ classifying the tangent bundle of $M$.
\end{definition}

\begin{theorem}
    [{\cite[Theorem~B]{rsp}}]
    \label{thm: rsp iff}
    Connected, closed, smooth $2k$-manifolds with $\pi$-finite tangential $(k-1)$-type are stably diffeomorphic if and only if they are indistinguishable by ordinary TQFTs.
\end{theorem}

Isomorphism clearly implies stable isomorphism. In the smooth case, Kreck showed in {\cite[Theorem~D]{kreck}} that, surprisingly, stable diffeomorphism also implies diffeomorphism for simply connected manifolds in dimension $4n+2$. Thus, the simply connected $(4n+2)$-dimensional version of the above theorem can be reformulated as follows.

\begin{proposition} [{\cite{rsp}}]
\label{prop: pi finite k odd}
    When $k$ is odd, ordinary TQFTs distinguish diffeomorphism classes of smooth, simply connected, closed $2k$-manifolds with $\pi$-finite tangential $(k-1)$-type. 
\end{proposition}

    It is worth noting that stable isomorphism is not strictly stronger or weaker than homotopy equivalence. On the one hand, there are families of smooth $4$-manifolds presented by Teichner in \cite{teichner} that have $\pi$-finite tangential $1$-type and are pairwise homotopy equivalent but pairwise not stably diffeomorphic. On the other hand, Kreck and Schafer observed in \cite{kreckandschafer} that Wall's work \cite{wall} implies that, for $n\geq 2$, there are families of $(2n-1)$-connected $4n$-manifolds that are pairwise stably diffeomorphic but are pairwise not homotopy equivalent. This is also true in $4$ dimensions in the topological category, which follows from Freedman's theorem \cite{freedmanquinn}. More examples of pairs of manifolds that are stably diffeomorphic but not homotopy equivalent were presented recently in \cite{conway1} and \cite{conway2}. So, ordinary TQFTs sometimes cannot see homotopy type, but they can sometimes see \emph{beyond} homotopy type.

In \cite{reutter}, Reutter showed that ordinary $4$-dimensional TQFTs cannot distinguish smooth structures. A theorem of Gompf \cite{gompf} states that any two smoothings of a compact, oriented $4$-manifold are stably diffeomorphic. Theorem \ref{thm: rsp A} (\cite[Theorem A]{rsp}) then implies that orientation-preserving homeomorphic $4$-manifolds cannot be distinguished by ordinary TQFTs. In particular, exotic $4$-spheres cannot be detected. On the other hand, Reutter and Schommer-Pries showed in \cite[Proposition~1.2]{rsp} that, in dimensions $8k+1$ and $8k+2$ for $k\geq1$, there are pairs of exotic spheres that can be distinguished. In even dimensions $2k>6$, homotopy spheres are diffeomorphic if and only if they are stably diffeomorphic \cite{rei}, but they do not have $\pi$-finite tangential $(k-1)$-type, so Theorem \ref{thm: rsp iff} is not applicable. This leads to Question 1.3 in \cite{rsp} which asks whether the $\pi$-finiteness condition for even-dimensions can be relaxed and whether exotic spheres in all dimension greater than $4$ can be distinguished. This question was partially answered by Gripaios and Randal-Williams, who proved in \cite{milnorsphere} that $(4k-1)$-dimensional homotopy spheres that bound parallelizable manifolds -- such as the Milnor $7$-sphere -- are indistinguishable to TQFTs. 

This paper was motivated in particular by the question of weakening the $\pi$-finiteness condition in $6$ dimensions. A connected, smooth $6$-manifold has $\pi$-finite tangential $2$-type if it has finite $\pi_1$ and $\pi_2$.

\begin{question}
\label{ques:6D}
Does Theorem \ref{thm: rsp iff} hold in $6$ dimensions if we drop the requirement that $\pi_2$ be finite? 
\end{question}

We answer Question \ref{ques:6D} by showing that the finiteness condition on $\pi_2$ cannot be dropped. We produce a pair of simply connected, smooth $6$-manifolds with infinite $\pi_2$ which are not diffeomorphic (and therefore not stably diffeomorphic by \cite[Theorem~D]{kreck}), but are nevertheless indistinguishable by ordinary TQFTs. This can be extended to higher dimensions. For all $m\geq6$, we construct a pair of simply connected, smooth $m$-manifolds which are neither diffeomorphic nor stably diffeomorphic (if $m$ is even) but are indistinguishable to ordinary TQFTs.

Since this construction arose from an attempt to weaken the $\pi$-finiteness condition in $6$ dimensions, it is natural to ask the following question.

\begin{question}
    \label{higher d?}
    Are there more highly connected analogues of the above construction that prove the necessity of the $\pi$-finiteness conditions in higher degrees?
\end{question}

We partially answer this question with the following generalization: For all $n\geq 2$ and $m\geq2n$, there are pairs of $(2n-1)$-connected smooth $(4n+m)$-manifolds which are neither diffeomorphic nor stably diffeomorphic (if $m$ is even) but are indistinguishable by ordinary TQFTs (see Section \ref{sec: high conn}).

\subsection{Summary of Results and Organization of the Paper}

In Section \ref{sec simpletqft}, we present a proof of Theorem \ref{thm: rsp A} based on that of \cite[Proposition~3.2]{reutter}.

In Section \ref{sec: Top simp conn}, we introduce pairs of non-homotopy equivalent topological $(4+m)$-manifolds that are simply connected for $m>1$ and cannot be distinguished by ordinary TQFTs. They are constructed using two simply connected topological $4$-dimensional "building block" manifolds denoted, for now, $M^4$ and $N^4$. 

\begin{propositionalpha}[Proposition \ref{prop:88 16 product not htpy eq or stab iso}]
\label{prop: intro man}
The manifolds $\ssm{M^4}$ and $\ssm{N^4}$ are not homotopy equivalent. Furthermore, for $m$ even and nonzero, they are still not homotopy equivalent after stabilization.\footnote{Stabilization refers to taking a finite number of connected sums with $\LL S^\ell \times S^\ell \RR$, where $\ell=\frac{1}{2}(4+m)$.}
\end{propositionalpha}

We note the obvious fact that Proposition \ref{prop: intro man} implies that $\ssm{M^4}$ and $\ssm{N^4}$ are neither isomorphic nor stably isomorphic (for $m$ even and nonzero). 

The manifolds $M^4$ and $N^4$ are not smoothable, so we first specialize to the topological category. By Freedman's Theorem \cite{freedmanquinn}, $M^4$ and $N^4$ are stably homeomorphic and, by Theorem \ref{thm: rsp A}, are thus indistinguishable to ordinary TQFTs. Using dimensional reduction along $S^m$, we prove our main result in the topological category.

\begin{theoremalpha}[Theorem \ref{thm: main thm top}]
    \label{thm: intro main top}
    Ordinary $(4+m)$-dimensional topological (non-smooth) TQFTs cannot distinguish $\ssm{M^4}$ and $\ssm{N^4}$.

\end{theoremalpha}

In Section \ref{sec:smoothing background}, we provide some background in smoothing theory. In \ref{ss:microbundles}, we review microbundles. In \ref{ss:KS background}, we discuss classifying spaces for microbundles and a major smoothing theory result of Kirby and Siebenmann (Theorem \ref{thm:KS10.1 sm version}). This result states that, in high dimensions, the smoothability of a topological manifold is equivalent to the existence of a vector bundle refinement of its tangent microbundle.

Section \ref{sec:FS} considers manifolds with bundle-theoretic smoothing data called formally smooth (FS) manifolds. High-dimensional FS manifolds are smoothable, but, since smoothing theory fails in low dimensions, there are $4$-manifolds which have FS structures but no smooth structures, including our $4$-dimensional building blocks referenced above. In \ref{ss: FS bord}, we discuss operations on and bordisms between FS manifolds. In \ref{ss:FS TQFTs}, we introduce a new bordism category of FS manifolds, denoted $\fscat{n}$. We show that our topological counter-examples $\ssm{M^4}$ and $\ssm{N^4}$ from Section \ref{sec: Top simp conn} have FS structures and extend our main theorem to the FS category.

\begin{theoremalpha}[Theorem \ref{thm: main thm FS}]
\label{thm: intro FS main}
      Ordinary $(4+m)$-dimensional FS TQFTs cannot distinguish $\ssm{M^4}$ and $\ssm{N^4}$.
\end{theoremalpha}

Section \ref{sec:smooth} contains two main results. First, we construct a functor \[\mathscr{F}_n:\smcat{n}\to \fscat{n}\] and show that, for $n\geq6$, it is essentially surjective and fully faithful, yielding the following result.

\begin{theoremalpha}[Theorem \ref{thm: bord equiv}]
    For $n\geq6$, the categories $\fscat{n}$ and $\smcat{n}$ are equivalent.
\end{theoremalpha}

Our $4$-manifold building blocks are FS but not smooth, so we cannot perform our dimensional reduction argument between smooth bordism categories. Since smooth structures may not be unique, there is no well-defined functor $-\times S^m$ from the topological bordism category to the smooth one. Thus, we dimensionally reduce in the FS category and then use a weak inverse of $\mathscr{F}_{4+m}$ to map to the smooth category: \[\fscat{4}\xrightarrow{-\times S^m}\fscat{n}\xrightarrow{\mathscr{F}_{4+m}^{-1}}\smcat{4+m}\to\vectc.\]

\noindent This allows us to extend our main result to the smooth case.

\begin{theoremalpha}[Theorem \ref{thm: sm cant dist htpy}]
      For all $m\geq2$, ordinary $(4+m)$-dimensional smooth TQFTs cannot distinguish $\ssm{M^4}$ and $\ssm{N^4}$.
\end{theoremalpha}

Section \ref{sec: dim 5} addresses the $5$-dimensional case. When reducing from a $5$-dimensional TQFT to a $4$-dimensional one, the only option is to reduce along $S^1$, so we cannot produce any simply connected counter-examples (or, more generally, with finite $\pi_1$). In fact, Kreck and Teichner in \cite[Theorem~3]{kt} showed that all simply connected closed smooth $5$-manifolds can be distinguished by a particular class of ordinary TQFTs, called positive unitary TQFTs. We show that, on the other hand, ordinary TQFTs cannot distinguish \emph{non}-simply connected $5$-manifolds.

\begin{theoremalpha}[Theorem \ref{thm: 5 dim smooth}]
    In the topological, FS, and smooth categories, ordinary $5$-dimensional TQFTs fail to distinguish $M^4\times S^1$ and $N^4\times S^1$.
\end{theoremalpha}

In the topological and FS categories, the $5$-dimensional situation is no different from those in higher dimensions. However, the smooth case is subtle since $\mathscr{F}_5$ is not essentially surjective. It is, however, fully faithful, which allows us to bypass $\fscat{5}$ by constructing a direct dimensional reduction functor $\fscat{4}\xrightarrow{-\times S^1}\smcat{5}$.

Section \ref{sec: non spin} begins with a brief discussion of spinnability for topological manifolds. We then show the following.

\begin{propositionalpha}[Corollaries \ref{cor: are spin}, \ref{cor: are not spin}, Proposition \ref{prop: non spin not htpy eq}]
\label{thm: intro nonspin}
    For all $m$, the following are true:
    \begin{enumerate}
        \item $\ssm{M^{4}}$ and $\ssm{N^{4}}$ are spinnable.
        \item $\ssm{\LL M^4 \# \mathbb{C}P^2\RR}$ and $\ssm{\LL N^4\# \mathbb{C}P^2\RR}$ are not spinnable.
        \item $\ssm{\LL M^4 \# \mathbb{C}P^2\RR}$ and $\ssm{\LL N^4\# \mathbb{C}P^2\RR}$ are not homotopy equivalent (or homotopy equivalent after stabilization if $m$ is even and nonzero), yet cannot be distinguished by ordinary TQFTs in the topological, FS, or smooth categories.
    \end{enumerate}
    
\end{propositionalpha}

In Section \ref{sec: high conn}, we discuss some more highly connected analogues of our counter-examples built from $M^4$ and $N^4$. For all $n\geq2$, there is a pair of stably-diffeomorphic, $(2n-1)$-connected, closed, smooth $4n$-manifolds denoted $M^{4n}$ and $N^{4n}$. Again using dimensional reduction, we prove the following result. The proof is identical in the topological, FS, and smooth categories.

\begin{propositionalpha}[Proposition \ref{prop: high conn 88 16 product not htpy eq or stab iso}, Proposition \ref{prop: high conn main theorem}]
    For all $n\geq2$ and all $m$, the manifolds $\ssm{M^{4n}}$ and $\ssm{N^{4n}}$ are not homotopy equivalent (or homotopy equivalent after stabilization, if $m$ is even and nonzero), but are nevertheless indistinguishable by ordinary TQFTs. For $m\geq2n$, they are $(2n-1)$-connected.
\end{propositionalpha}

The \ref{appendix}ppendix contains Magma code that was used to verify that $\ssm{\LL M^4 \# \mathbb{C}P^2\RR}$ and $\ssm{\LL N^4\# \mathbb{C}P^2\RR}$ are not homotopy equivalent.

\subsection{Unitary and Once-Extended TQFTs are Ordinary}

For a manifold $M$, let $\overline{M}$ denote the manifold with opposite orientation. If $M:X\rightsquigarrow Y$ is a bordism, flipping the orientation yields a bordism $\overline{M}: Y\rightsquigarrow X$.

\begin{definition}
    A TQFT is \textbf{unitary} if it factors through the category $\hilb$ of complex Hilbert spaces and satisfies the condition that, for any bordism $M:X \rightsquigarrow Y$, the linear map $\mathcal{Z}(\overline{M}):\mathcal{Z}(Y)\to \mathcal{Z}(X)$ is the adjoint of $\mathcal{Z}(M):\mathcal{Z}(X)\to\mathcal{Z}(Y)$, denoted $\mathcal{Z}(M)^\dagger$.
\end{definition}

\begin{lemma}
    Dimensional reductions of unitary TQFTs are unitary.
    \begin{proof}
        Let $\mathcal{Z}:\bord_n\to \vectc$ be a unitary TQFT, and let $\mathcal{W}$ be the dimensional reduction of $\mathcal{Z}$ along some closed $(n-q)$-manifold $K$. Let $M$ be a $q$-dimensional bordism. Then, $\mathcal{W}(\overline{M})=\mathcal{Z}(\overline{M}\times K)=\mathcal{Z}(\overline{M\times K})$, since $\overline{M}\times K$ and $ \overline{M\times K}$ are orientation-preserving isomorphic. Since $\mathcal{Z}$ is unitary, $\mathcal{Z}(\overline{M\times K})=\mathcal{Z}(M\times K)^\dagger=\mathcal{W}(M)^\dagger$. Thus, $\mathcal{W}$ is unitary.
    \end{proof}
\end{lemma}

\begin{proposition}
\label{prop: unitary are ord}
    Unitary TQFTs are ordinary.
    \begin{proof}
     Let $\mathcal{Z}$ be a unitary $n$-dimensional TQFT, and let $\mathcal{W}$ be the dimensional reduction of $\mathcal{Z}$ along some closed $(n-q)$-manifold $M$. By the above lemma, $\mathcal{W}$ is unitary. The proof of Theorem 2.9 in \cite{reutter} shows that any unitary $q$-dimensional TQFT maps $S^{q-1}$ to a semisimple algebra. Thus, $\mathcal{W}(S^{q-1})=\mathcal{Z}(M\times S^{q-1})$ is semisimple.
    \end{proof}
\end{proposition}

A once-extended TQFT is a symmetric monoidal $2$-functor from the bicategory of bordisms to the bicategory of complex $2$-vector spaces (see \cite{christhesis} for a precise definition). The proof of Theorem 2.10 in \cite{reutter} shows that any once-extended $q$-dimensional TQFT maps $S^{q-1}$ to a semisimple algebra. An argument analogous to the proof of Proposition \ref{prop: unitary are ord} proves the following result.

\begin{proposition}
    Once-extended TQFTs are ordinary. \qed
\end{proposition}

\begin{remark}
    In dimension $\geq 4$, all currently known examples of TQFTs are either unitary or once-extended, and are thus ordinary. 
\end{remark}

\subsection{Acknowledgments}
I would like to thank my PhD advisor, Christopher Schommer-Pries, for his consistent guidance and encouragement and for presenting me with this problem. I am also deeply grateful to Stephan Stolz for his support and many invaluable discussions. I would also like to thank Alexander Kupers for proposing the use of smooth microbundles in this paper and Lorenzo Riva for helpful comments and suggestions. This work was partially supported by the National Science Foundation under Grant No. DMS-2204297. Occasional discussions with ChatGPT were  used to improve understanding of existing literature and mathematical concepts related to the project as well as to debug \LaTeX \hspace{.5mm}code in some parts of this document.


\section{Simple TQFTs and Stable Isomorphism} \label{sec simpletqft}

\begin{definition}
    An $n$-dimensional TQFT $\mathcal{Z}$ is \textbf{simple} if $\mathcal{Z}(S^{n-1})\cong \mathbb{C}$.
\end{definition}
For $n\geq 2$, the value of an $n$-dimensional ordinary TQFT on $S^{n-1}$ is commutative and semisimple, so it is a direct sum of copies of $\mathbb{C}$ by the Artin-Wedderburn theorem. By Corollary 5.11 of \cite{rsp}, any ordinary TQFT in dimension $\geq2$ can be decomposed as a finite direct sum of simple theories, which are multiplicative over connected sums in the following sense.

\begin{proposition}[{\cite[Proposition~3.2]{reutter}}]
\label{prop:mult under conn sum}
    Let $M$ and $N$ be connected $n$-dimensional manifolds with $M$ closed, and let $\mathcal{Z}$ be a simple $n$-dimensional TQFT. Then,
    \begin{equation*}
        \mathcal{Z}(M\#N)=\LL \mathcal{Z}(S^n)\RR^{-1}\mathcal{Z}(M)\mathcal{Z}(N).
    \end{equation*}
    \begin{proof}
        By removing open $n$-disks from $M$ and $N$ (away from the boundary), we define the bordisms $M'=M\setminus \Int (D^n):\varnothing\rightsquigarrow S^{n-1}$ and $N'=N\setminus \Int (D^n): \partial N \rightsquigarrow S^{n-1}$. We will denote the bordism $D^n:S^{n-1}\rightsquigarrow\varnothing$ by $\varepsilon$ since it corresponds to the counit of the Frobenius algebra $\mathcal{Z}\LL S^{n-1}\RR$. We can then decompose $M$ and $N$ as \[M : \varnothing \overset{M'}\longrightsquigarrow S^{n-1} \overset{\varepsilon}{\longrightsquigarrow} \varnothing\] and \[N : \partial N \overset{N'}{\longrightsquigarrow} S^{n-1} \overset{\varepsilon}{\longrightsquigarrow} \varnothing.\]
     \noindent We may decompose the connected sum $M\#N$ as
     \begin{equation*}
         M\#N: \varnothing \sqcup \partial N \overset{M'\sqcup N'}{\extralongrightsquigarrow} S^{n-1} \sqcup S^{n-1} \overset{\mu}{\longrightsquigarrow} S^{n-1} \overset{\varepsilon}{\longrightsquigarrow} \varnothing, 
     \end{equation*}

\noindent as is depicted below.

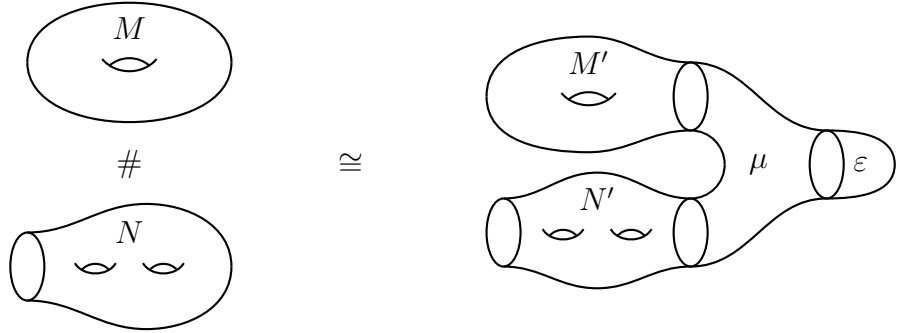
\begin{center}
\begin{tikzpicture} [scale=0.9]
\draw[thick, shift=({-0.25,0.5})] (-6.5,1) to[out=-90, in=-90] (-3.5,1);
\draw[thick, shift=({-0.25,0.5})] (-6.5,1) to[out=90, in=90] (-3.5,1);
\draw[thick, shift=({-0.25,0.5})] (-5.4,1.05) to[out=-50, in=-130] (-4.6,1.05);
\draw[thick, shift=({-0.25,0.5})] (-5.3,0.95) to[out=40, in=140] (-4.7,0.95);

\node at (-5.25,2) {$M$};
\node at (-5.25,-1) {$N$};
\node at (-5.25,0) {$\#$};

\draw[thick] (-6.75,-1.5) ellipse(0.25cm and 0.5cm);
\draw[thick, shift=({-0.25,-0.5})] (-6.5,-1.5) to[out=0, in=180] (-4.75,-1.92) to[out=0, in=-90] (-3.5,-1);
\draw[thick, shift=({-0.25,-0.5})] (-6.5,-0.5) to[out=0, in=180] (-4.75, -0.08) to[out= 0, in=90] (-3.5,-1);
\draw[thick, shift=({-0.25,-0.5})] (-5.8,-0.95) to[out=-60, in=-120] (-5.2,-0.95);
\draw[thick, shift=({-0.25,-0.5})] (-5.7,-1.05) to[out=50, in=130] (-5.3,-1.05);
\draw[thick, shift=({-0.25,-0.5})] (-4.8,-0.95) to[out=-60, in=-120] (-4.2,-0.95);
\draw[thick, shift=({-0.25,-0.5})] (-4.7,-1.05) to[out=50, in=130] (-4.3,-1.05);

\node at (-2,0) {$\cong$};

\draw[thick] (0,1) to[out=90, in=180] (1.5,1.88)to[out=0, in=180] (3,1.5);
\draw[thick] (0,1) to[out=-90, in=180] (1.5,0.18) to[out=0, in=180] (3,0.5) to[out=0, in=90] (3.5,0);
\draw[thick] (1.1,1.05) to[out=-50, in=-130] (1.9,1.05);
\draw[thick] (1.2,0.95) to[out=40, in=140] (1.8,0.95);

\node at (1.5,1.5) {$M'$};
\node at (1.625,-0.5) {$N'$};
\node at (4,0) {$\mu$};
\node at (5.5,0) {$\varepsilon$};

\draw[thick] (0.250,-1) ellipse(0.25cm and 0.5cm);
\draw[thick] (0.25,-0.5) to[out=0, in=180] (1.625,-0.12)to[out=0, in=180] (3,-0.5) to[out=0, in =-90] (3.5,0);
\draw[thick] (0.25,-1.5) to[out=0, in=180] (1.625,-1.82) to[out=0, in=180] (3,-1.5);

\draw[thick] (3,1) ellipse(0.25cm and 0.5cm);
\draw[thick] (3,-1) ellipse(0.25cm and 0.5cm);
\draw[thick] (5,0) ellipse(0.25cm and 0.5cm);
\draw[thick] (3,1.5) to[out=0, in=180] (5,0.5);
\draw[thick] (3,-1.5) to[out=0, in=180] (5,-0.5);
\draw[thick] (5,0.5) to[out=0, in=90] (6,0) to[out=-90, in=0] (5,-0.5);

\draw[thick] (0.825,-0.95) to[out=-60, in=-120] (1.425,-0.95);
\draw[thick] (0.925,-1.05) to[out=50, in=130] (1.325,-1.05);
\draw[thick] (1.825,-0.95) to[out=-60, in=-120] (2.425,-0.95);
\draw[thick] (1.925,-1.05) to[out=50, in=130](2.325,-1.05);
\end{tikzpicture}
\end{center} 

     Applying $\mathcal{Z}$ to $M\sqcup N$ and $M\#N$ yields the following linear maps: \[\mathcal{Z}(M\sqcup N): \mathbb{C} \otimes\mathcal{Z} \LL \partial N \RR  \xlongrightarrow{\mathcal{Z}(M')\otimes \mathcal{Z}(N')} \mathcal{Z}(S^{n-1}) \otimes \mathcal{Z}(S^{n-1}) \xlongrightarrow{\mathcal{Z}(\varepsilon)\otimes \mathcal{Z}(\varepsilon)} \mathbb{C},\] \[\mathcal{Z}(M\#N): \mathbb{C} \otimes\mathcal{Z} \LL \partial N \RR\xlongrightarrow{\mathcal{Z}(M')\otimes \mathcal{Z}(N')} \mathcal{Z}(S^{n-1}) \otimes \mathcal{Z}(S^{n-1}) \xlongrightarrow{\mathcal{Z(\mu)}}\mathcal{Z}\LL S^{n-1}\RR \xlongrightarrow{\mathcal{Z}(\varepsilon)} \mathbb{C}. \] It now remains to show that $\mathcal{Z}(\varepsilon)\otimes \mathcal{Z}(\varepsilon)$ differs from the composite \[\mathcal{Z}(S^{n-1}) \otimes \mathcal{Z}(S^{n-1}) \xlongrightarrow{\mathcal{Z(\mu)}}\mathcal{Z}\LL S^{n-1}\RR \xlongrightarrow{\mathcal{Z}(\varepsilon)} \mathbb{C}\] by a factor of $\mathcal{Z}(S^n)$. Since $\mathcal{Z}\LL S^{n-1}\RR$ is 1-dimensional, the unit $\mathcal{Z}(\upsilon)$ is invertible, so the isomorphisms of bordisms depicted below complete the proof.

     
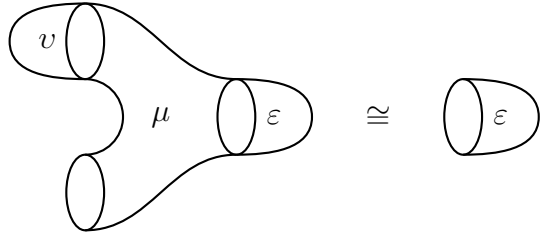
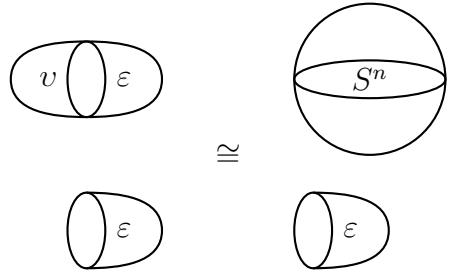
\begin{center}
         \begin{tikzpicture}
            
\node at (6.875,0) {$\cong$};
\node at (6.875,-5) {$\cong$};

\node at (2.5,1) {$\upsilon$};
\node at (4.5,-4) {$\upsilon$};

\node at (4,0) {$\mu$};

\node at (5.5,0) {$\varepsilon$};
\node at (8.5,0) {$\varepsilon$};
\node at (8.5,-6) {$\varepsilon$};
\node at (5.5,-6) {$\varepsilon$};
\node at (5.5,-4) {$\varepsilon$};

\node at (8.75,-4) {$S^n$};

\draw[thick] (3,1) ellipse(0.25cm and 0.5cm);
\draw[thick] (3,-1) ellipse(0.25cm and 0.5cm);
\draw[thick] (5,0) ellipse(0.25cm and 0.5cm);
\draw[thick] (5,0.5) to[out=0, in=90] (6,0) to[out=-90, in=0] (5,-0.5);
\draw[thick] (3,1.5) to[out=0, in=180] (5,0.5);
\draw[thick] (3,-1.5) to[out=0, in=180] (5,-0.5);
\draw[thick] (3,-0.5) to[out=0, in=-90] (3.5,0) to[out=90, in=0] (3,0.5);

\draw[thick, shift=({7,4})] (-4,-2.5) to[out=180, in=90] (-5,-3) to[out=-90, in=180] (-4,-3.5);

\draw[thick, shift=({3,0})] (5,0) ellipse(0.25cm and 0.5cm);
\draw[thick,shift=({3,0})] (5,0.5) to[out=0, in=90] (6,0) to[out=-90, in=0] (5,-0.5);

\draw[thick, shift=({9,-1})] (-4,-2.5) to[out=180, in=90] (-5,-3) to[out=-90, in=180] (-4,-3.5);
\draw[thick, shift=({0,-4})] (5,0) ellipse(0.25cm and 0.5cm);
\draw[thick,shift=({0,-4})] (5,0.5) to[out=0, in=90] (6,0) to[out=-90, in=0] (5,-0.5);
\draw[thick, shift=({0,-6})] (5,0) ellipse(0.25cm and 0.5cm);
\draw[thick,shift=({0,-6})] (5,0.5) to[out=0, in=90] (6,0) to[out=-90, in=0] (5,-0.5);

\draw[thick,shift=({3.75,-4})] (5,0) circle(1 cm);
\draw[thick,shift=({3.75,-4})] (5,0) ellipse(1 cm and 0.25 cm);
\draw[thick, shift=({3,-6})] (5,0) ellipse(0.25cm and 0.5cm);
\draw[thick,shift=({3,-6})] (5,0.5) to[out=0, in=90] (6,0) to[out=-90, in=0] (5,-0.5);
\end{tikzpicture}
     \end{center}
     \end{proof}
\end{proposition}

\begin{proposition}
\label{prop: torus invert}
     If $\mathcal{Z}$ is a simple TQFT, $\mathcal{Z}(S^k\times S^k)$ is invertible.
     \begin{proof}
       Reutter and Schommer-Pries showed (\cite[proof of Proposition 5.17]{rsp}) that the bordisms \[\LL W_{S^k}\times S^{k-1}\RR\circ\LL S^k\times D^k\RR: \varnothing \rightsquigarrow S^k\times S^{k-1} \] 
       \noindent and 
       \[\LL S^1\times S^{2k-1}\RR\#\LL S^k\times S^k\RR\#\LL S^k\times D^k\RR:\varnothing \rightsquigarrow S^k\times S^{k-1}\]
       \noindent are isomorphic, where $W_{S^k}:=\LL S^k \times S^1\RR \setminus \LL D^{k+1}\sqcup D^{k+1}\RR$. 
       
       Thus, 
             \[\mathcal{Z}\LL W_{S^k}\times S^{k-1}\RR \mathcal{Z}\LL S^k\times D^k\RR = \LL \mathcal{Z}\left(S^{2k}\right)\RR ^{-2}\mathcal{Z}\LL S^1\times S^{2k-1}\RR \mathcal{Z}\LL S^k\times S^k\RR \mathcal{Z}\LL S^k\times D^k\RR.\]

Hence, $\mathcal{Z}\LL S^k\times D^k\RR$ is an eigenvector with eigenvalue \[\LL \mathcal{Z}\LL S^{2k}\RR \RR ^{-2}\mathcal{Z}\LL S^1\times S^{2k-1}\RR \mathcal{Z}\LL S^k\times S^k\RR.\] By \cite[Proposition~3.3]{reutter} the map $\mathcal{Z}\LL W_{S^k}\times S^{k-1}\RR$ is an invertible endomorphism of the algebra $\mathcal{Z}(S^k\times S^{k-1})$. Thus, all factors of the eigenvalue, in particular $\mathcal{Z}\LL S^k \times S^k\RR$, are invertible.
     \end{proof}
\end{proposition}
It follows from Propositions \ref{prop:mult under conn sum} and \ref{prop: torus invert} that stably isomorphic closed manifolds are indistinguishable by simple TQFTs, and thus by ordinary TQFTs in dimension $\geq 2$: 

\begin{proposition}
[{\cite[Proposition~5.12]{rsp}}]
\label{prop: cant dist stab iso}
    Even-dimensional ordinary TQFTs cannot distinguish stably isomorphic closed manifolds. \qed
\end{proposition}


\section{Topological Results for Infinite $\bm{\pi_2}$}
\label{sec: Top simp conn}

Let $E_8$ denote the unique positive-definite, even, unimodular lattice of rank 8, and let $E_{16}$ denote the unique positive-definite, even, unimodular lattice of rank 16 that is not isometric to $E_8\oplus E_8$. By Freedman's theorem \cite{freedmanquinn} There are unique (up to homeomorphism) closed, simply connected topological 4-manifolds, denoted $\meight$ and $\msixt$, with intersection forms\footnote{We define the intersection form of a closed $2k$-manifold $M$ to be the bilinear form on $H^k(M)\times H^k(M)$ given by $(\alpha,\beta)\mapsto\langle \alpha\smile\beta,[M]\rangle$, where $[M]$ is the fundamental class of $M$.} $E_8\oplus E_8$ and $E_{16}$, respectively. Note that $\meight$ can be realized by taking the connected sum of two copies of $M^4_{E_8}$, the unique (up to homeomorphism) closed, simply connected topological $4$-manifold with intersection form $E_8$. 

\begin{proposition}
\label{prop:can't dist 88 and 16}
    Ordinary $4$-dimensional topological (non-smooth) TQFTs cannot distinguish $\meight$ and $\msixt$.
    \begin{proof}
        \sloppy Intersection forms are additive over the connected sum operation, so $\meight\#\LL S^2\times S^2\RR$ and $\msixt\#\LL S^2\times S^2\RR$ have intersection forms $E_8\oplus E_8 \oplus H$ and $E_{16}\oplus H$, respectively, where $H$ is the hyperbolic form on $\mathbb{Z}^2$. Both lattices are even, indefinite, and have the same rank and signature, and are thus isometric. By Freedman's Theorem \cite{freedmanquinn}, $\meight\#\LL S^2\times S^2\RR \cong \msixt\#\LL S^2\times S^2\RR$. Thus, $\meight$ and $\msixt$ are stably homeomorphic and are indistinguishable by ordinary $4$-dimensional topological TQFTs, by Proposition \ref{prop: cant dist stab iso}.
    \end{proof}
\end{proposition}

\begin{proposition}
    \label{prop:88 16 product not htpy eq or stab iso}
    The manifolds $\meight\times S^m$ and $\msixt\times S^m$ are not homotopy equivalent. If $m$ is even and nonzero they are also not homotopy equivalent after stabilization. 
\end{proposition}

The proof of Proposition \ref{prop:88 16 product not htpy eq or stab iso} follows immediately from the two lemmas below.

\begin{lemma}
    \label{lem: diff int not htpy eq after cross}
    Let $M_1$ and $M_2$ be $(k-1)$-connected $2k$-manifolds with intersection forms $q_1$ and $q_2$, where $q_1$ is not isometric to $\pm q_2$. For $m\neq k$, $\ssm{M_1}$ and $\ssm{M_2}$ are not homotopy equivalent. The statement also holds for $m=k$ if $q_1$ and $q_2$ are definite.
   
    \begin{proof}
      If $m\neq k$, $S^m$ contributes nothing to $H^k(\ssm{M_i})$, so the result is immediate. If $m=k$, then, $S^k$ contributes a $\mathbb{Z}$ summand to $H^k\LL \ssm{M_i}\RR$, generated by some class $\sigma_i$, and we have canonical splittings $H^k\LL \ssm{M_i}\RR\cong H^k(M_i)\oplus \mathbb{Z}[\sigma_i]$. If $q_i$ is definite, $\mathbb{Z}[\sigma_i]$ is the only cup-square-zero subgroup of $H^k\LL \ssm{M_i}\RR$. Consider the trilinear form $Q_{M_i}$ on $H^k\LL \ssm{M_i}\RR$ given by evaluating the triple cup product on the fundamental class of $\ssm{M_i}$. This form descends to $q_i$ on the quotient $H^k(\ssm{M_i})/\mathbb{Z}[\sigma_i]\cong H^k(M_i)$. Thus, if $q_1$ is not isometric to $\pm q_2$, $Q_1$ is not isometric to $\pm Q_2$. Any homotopy equivalence $\ssm{M_1}\to \ssm{M_2}$ would yield an isometry from $Q_{M_1}$ to $\pm Q_{M_2}$, which completes the proof.
        \end{proof}
\end{lemma}
\begin{lemma}
    \label{lem: diff int not stab htp}
    Let $X_1$ and $X_2$ be non homotopy equivalent $2k$-manifolds. If $i<k$ and there are multilinear forms on $H^i(X_1)$ and $H^i(X_2)$ witnessing the inequivalence (e.g. $Q_{M_1}$ and $Q_{M_2}$ in the lemma above), then $X_1$ and $X_2$ are still not homotopy equivalent after finitely many connected sums with $S^k\times S^k$.
    \begin{proof}
        Since $k>i$, taking repeated connected sums with $S^k\times S^k$ has no effect on $H^i$.
    \end{proof}
\end{lemma}

\begin{theorem}
    \label{thm: main thm top}
    Ordinary $(4+m)$-dimensional topological (non-smooth) TQFTs fail to distinguish $\ssm{\meight}$ and $\ssm{\msixt}$.
    \begin{proof}
        Let $\mathcal{Z}$ be an ordinary $(4+m)$-dimensional topological TQFT and let $\mathcal{W}$ be the dimensional reduction of $\mathcal{Z}$ along $S^m$: \[\mathcal{W}: \topcat{4} \xlongrightarrow{-\times S^m}\topcat{4+m}\xlongrightarrow{\mathcal{Z}} \vectc.\]
\noindent By Proposition \ref{prop:can't dist 88 and 16}, $\mathcal{W}(\meight)=\mathcal{W}(\msixt)$, so $\mathcal{Z}(\ssm{\meight}) = \mathcal{Z}(\ssm{\msixt})$. 
\end{proof}
\end{theorem}


\section{Some Background in Smoothing Theory}
\label{sec:smoothing background}

\subsection{Microbundles}
\label{ss:microbundles}

Topological microbundles, a generalization of vector bundles, were introduced by Milnor in \cite{milnorpart1} in order to define an analogue of tangent bundles for topological manifolds.

\begin{definition}
 [{\cite{milnorpart1}}]
\label{def: top micro}
 A \textbf{topological} $\bm{k}$\textbf{-microbundle} $\scriptstyle \mathfrak{X}$ is a quadruple $(E,B,
 p,s)$, where $E$ and $B$ are topological spaces, and $p: E \rightarrow  B$ and $s:B \rightarrow  E$ are continuous maps subject to the following conditions.
    \begin{enumerate}
        \item The composition $ps$ is the identity on $B$.
        \item Local triviality: For every $b\in B$, there are open neighborhoods $U\subset B$ of $b$ and $V \subset E$ of $s(b)$ such that $p(V)\subset U$, $s(U)\subset V$, and there is a homeomorphism $V\rightarrow U\times \mathbb{R}^k$ such that the following diagram commutes, where $p_1$ is the projection onto the first factor and $\times \{0\}$ is the map $u\mapsto(u,0)$.
\begin{center}        
 \begin{tikzcd}
  &U \\
V \arrow{ur}{p|_V} \arrow{rr} && U\times\mathbb{R}^k \arrow[swap]{ul}{p_1} \\
&U \arrow{ul}{s|_U} \arrow[swap]{ur}{\times \{0\}}
\end{tikzcd}
 \end{center}      
    \end{enumerate}
We call $E$ the total space and $B$ the base space, and we will generally identify $B$ with $s(B)$.
\end{definition}

\begin{definition}[{\cite{milnorpart1}}]
    \label{def: micro restriction and pullback}
    Let $\exx = \LL E,B,p,s\RR$ be a topological microbundle.
    \begin{enumerate}
        \item If $A\subset B$, the \textbf{restriction} of $\exx$ to $A$, denoted $\exx|_A$ is $\LL p^{-1}(A),A,p\big|_{p^{-1}(A)},s\big|_{A}\RR$.
        \item If $f:B'\to B$ is a continuous map, the \textbf{pullback} of $B$ along $f$, denoted $f^*\exx$, is $\LL E', B', p',s'\RR$, where $E'=\{(b',e)\in B'\times E : f(b')=p(e)\}$, $p'(b',e)=b'$, and $s'(b')=(b',sf(b'))$.
    \end{enumerate}
    It is easy to verify that $\exx|_A$ and $f^*\exx$ are indeed topological microbundles.
\end{definition}

\begin{definition}[{\cite{milnorpart1}}]
\label{def: whitney sums} 
Let $\exx = \LL E,B,p,s \RR$ and $\exx' = \LL E',B',p',s' \RR$ be topological vector bundles. The \textbf{external Whitney sum}\footnote{In \cite{milnorpart1}, this is called the Cartesian product and is denoted $\exx \times \exx'$.} of $\exx$ and $\exx '$ is $\LL E\times E', B\times B', p\times p', s\times s'\RR$, denoted by $\exx \boxplus \exx'$. It is again easy to verify that $\exx \boxplus \exx'$ is a topological microbundle. If $B=B'$, we define the ordinary \textbf{Whitney sum}, denoted $\exx \oplus \exx'$, to be the pullback $\Delta^* \LL \exx \boxplus \exx'\RR$, where $\Delta$ is the diagonal map $x\mapsto (x,x)$.
     
\end{definition}

\begin{definition}
    [{\cite{milnorpart1}}] Let $X$ be a topological $n$-manifold without boundary. The \textbf{topological tangent microbundle of} $\bm{X}$ is the topological $n$-microbundle $\micro{X}:=(X\times X,X,p_1,\Delta)$. 
\end{definition}

Local triviality is shown using coordinate charts homeomorphic to $\mathbb{R}^n$, so we cannot immediately extend this definition to manifolds with boundary. Instead, we use the following construction.

\begin{definition}
    Let $X$ be a topological manifold with boundary. We define $\bm{\widehat{X}}$ to be $X$ with a half-open collar attached to the boundary: $\widehat{X} = X \cup_{\partial X\times \{0\}}\left(\partial X \times [0,1)\right)$.
\end{definition}

\begin{definition}[{\cite[Essay IV, 10.2]{ks}}]
\label{def: deltawithbound}
   \sloppy The \textbf{tangent microbundle of a manifold with boundary} $X$ is defined as $\micro{X}:=\micro{\widehat{X}}\big|_X$. Note that the total space of $\micro{X}$ is equal to $X\times \widehat{X}$.
\end{definition}

\begin{definition}
\label{def: micro morphism}
\sloppy  A \textbf{topological microbundle morphism}\footnote{This is a more general notion than what is presented by Milnor and Kirby-Siebenmann. They use definitions analogous to what we call "$k$-microbundle maps" (see Definition \ref{def: micro map}).} from $(E,B,p,s)$ to $(E',B',p',s')$ is an equivalence class of triples $\bm{F}=\LL W, W', F \RR$, where $W \subset E$ and $W' \subset E'$ are open neighborhoods of $B$ and $B'$, respectively, and $(F,f):\LL W,B\RR \to \LL W',B'\RR$ is a continuous map making the following diagram commute.

 \begin{center}        
 \begin{tikzcd}
 \\
 & B \arrow{r}{f}& B' \\
W \arrow{ur}{p} \arrow{rrr}{F} &&& W' \arrow[swap]{ul}{p'} \\
&B \arrow{ul}{s}  \arrow{r}{f} &B'\arrow[swap]{ur}{s'}
\\
\end{tikzcd}
 \end{center}     
Two such triples $\left(W_1,W_1',F_1\right)$ and $\left(W_2,W'_2,F_2\right)$ are equivalent if there exists another triple $\left(W,W',F\right) $ such that $W\subset W_i$, $W'\subset W'_i$, and $F_i \big|_W=F$. We say $\bm{F}$ is a microbundle morphism over or covering $f$. Note that a microbundle isomorphism has a representative such that $(F,f)$ is a homeomorphism. 
\end{definition}

\begin{remark}
\label{rem:restrict micro iso}
  \sloppy  A topological microbundle $(E,B,p,s)$ is isomorphic to $(V,B,p\big|_V,s)$, where $V\subset E$ is any open neighborhood of $s(B)$.
\end{remark}

\begin{definition}
    \label{def: induced map on tan}
   \sloppy A continuous map between manifolds $f:X_1\to X_2$ induces a natural microbundle morphism $\bm{D}f:\micro{X_1}\to\micro{X_2}$ over $f$, which we call the \textbf{topological micro-differential}. If $X_1$ and $X_2$ have no boundary, $\bm{D}f$ is represented by $\LL X_1\times X_1,X_2\times X_2, f\times f\RR$. If they do have boundary, then there is a canonical map $\widehat{f}:\widehat{X_1}\to\widehat{X_2}$ which induces $\bm{D}\widehat{f}:\micro{\widehat{X_1}}\to\micro{\widehat{X_2}}$. Then, $\bm{D}f=\bm{D}\widehat{f}|_{\micro{X_1}}$.
\end{definition}

\begin{definition}[{\cite{milnorpart1}}{\cite{ks}}]
\label{def: micro map}
    A \textbf{topological} $\bm{k}$\textbf{-microbundle map} between rank $k$ microbundles is a microbundle morphism that gives open embeddings on fibers.
\end{definition}

The standard way to adapt the category of microbundles to a smooth setting is to replace topological spaces and continuous maps in the definitions above with smooth manifolds and smooth maps. However, it will be useful in our case to weaken this definition to allow for a notion of smooth microbundles over arbitrary topological spaces. Our definition mimics that of vector bundles over topological spaces, whose transition maps are linear in fiber directions. Smooth microbundles, defined precisely below, are topological microbundles whose transition maps are smooth in fiber directions. 

\begin{definition}[{\cite{lurienotes}}]
    \label{def: smooth micro}
   \sloppy A \textbf{smooth} $\bm{k}$\textbf{-microbundle} over a topological space $B$ is a topological $k$-microbundle $\LL E,B,p,s\RR$ admitting a maximal atlas of charts around $s(B)$ meeting the following compatibility condition. Let $b$ be a point in $B$, $U\subset B$ an open neighborhood of $b$, and $V,V'\subset E$ locally trivial neighborhoods over $U$ with trivializations $f:V\xrightarrow{\cong} U\times \mathbb{R}^k$ and $f':V'\xrightarrow{\cong} U\times \mathbb{R}^k$. Pick an $\varepsilon>0$ such that $f^{-1}(U\times B^k_\varepsilon)\subset V'$, where $B^k_\varepsilon$ is the open ball of radius $\varepsilon$ centered at the origin. We may identify $\mathbb{R}^k$ with $B^k_\varepsilon$. Then, for all $b\in U$, $ f'\circ f$ gives an origin-preserving embedding $\mathbb{R}^k\to \mathbb{R}^k$. We require that this embedding be smooth for each $b\in U$. 
\end{definition}

\begin{definition}
    A \textbf{smooth microbundle morphism} is a topological microbundle morphism between the underlying topological microbundles that gives smooth maps on fibers. \textbf{Smooth} $\bm{k}$\textbf{-microbundle maps} are defined analogously. 
\end{definition}

\begin{remark}
\label{rem: induced by smooth is smooth}
    It is straightforward to show that restrictions and pullbacks of smooth microbundles are again smooth microbundles.
\end{remark}

\begin{definition}
    If $X$ is a manifold with smooth structure $\Sigma$, the \textbf{smooth tangent microbundle}, denoted $\tau_{X_\Sigma}$, is the smooth microbundle given by $\micro{X}$ together with the atlas of charts determined by the product smooth structure on $X\times X$.
\end{definition}

\begin{remark}
     If $f:X_1\to X_2$ is a smooth map between smooth manifolds, then the topological micro-differential defined in Definition \ref{def: induced map on tan} is a smooth microbundle morphism.
\end{remark}

\begin{definition} [{\cite{milnorpart1}},{\cite[Example~5]{lurienotes}}]
    Let $E\xrightarrow{p} B$  be a rank $k$ vector bundle. Taking $s:B\rightarrow E$ to be the zero-section, $(E,B,p,s)$ is a topological $k$-microbundle over $B$, called the \textbf{underlying topological microbundle} of the vector bundle. If we add the data of the atlas of local trivializations of the vector bundle, we get a smooth microbundle, which we call the \textbf{underlying smooth microbundle}.
\end{definition}

\begin{proposition}[{\cite[Example~5]{lurienotes}}]
\label{prop: sm mb iso to vb}
    Let $B$ be paracompact. Then, there is a bijection between isomorphism classes of smooth $k$-microbundles over $B$ and isomorphism classes of rank $k$ vector bundles over $B$.
    \begin{proof}

Let $\LL E,B,p,s\RR$ be a smooth microbundle. By Remark \ref{rem:restrict micro iso}, we may assume that $E$ is entirely covered by an atlas of charts meeting the smoothness condition laid out in Definition \ref{def: smooth micro}. We define the relative tangent bundle $T_{E/B}$ over $E$ as follows: The atlas of charts gives us a family of smooth structures on the microbundle fibers which vary continuously with respect to $b\in B$. For $e\in E$, let $T_{E/B}\big|_e$ be the space of derivations at $e$ of smooth functions on $p^{-1}(p(e))$, and let $T_{E/B}={\bigsqcup_{e\in E}}T_{E/B}\big|_e$, equipped with the obvious projection map. The atlas of charts on $E$ determines a family of local trivializations whose transition functions are continuous and are linear isomorphisms on fibers, yielding a vector bundle structure on $T_{E/B}$. The pullback $s^*T_{E/B}$ is then a vector bundle over $B$. The map $f:(E,B,p,s)\mapsto s^*T_{E/B}$ determines a bijection on isomorphism classes: 

\emph{Well-Defined}: If two smooth microbundles over $B$ are isomorphic, the pushforward gives a vector bundle isomorphism on the images under $f$.

\emph{Surjective}: Let $E\xrightarrow{p}B$ be a vector bundle with underlying smooth microbundle $\exx$. Then, $f(\exx)$ is canonically isomorphic to $E\xrightarrow{p}B$.

\emph{Injective}: Let $\LL E,B,p,s\RR$ be a smooth microbundle. Since $B$ is paracompact, we may choose a continuously varying (with respect to $B$) family of Riemannian metrics on each fiber of $p$. This determines a smooth microbundle isomorphism $\bm{F}:s^*T_{E/B}\to \LL E,B,p,s\RR$ over the identity by taking the exponential map on each fiber. Let $\LL E',B,p',s'\RR$ be another smooth microbundle over $B$. We may define an analogous smooth microbundle isomorphism $\bm{F}':(s')^*T_{E'/B}\to\LL E',B,p',s'\RR$. If there is a vector bundle isomorphism $s^*T_{E/B}\xrightarrow{\cong}(s')^*T_{E'/B}$, then, since vector bundle isomorphisms are smooth microbunde isomorphisms if we forget linear structure, we may construct the following smooth microbundle isomorphism: \[\LL E,B,p,s\RR\xrightarrow{\bm{F}^{-1}}s^*T_{E/B}\xrightarrow{\cong}(s')^*T_{E'/B}\xrightarrow{\bm{F}'}\LL E',B,p',s'\RR.\]
    \end{proof}
\end{proposition} 

\emph{Conventions:}
We use lowercase Fraktur letters to denote topological microbundles and lower case Greek letters to denote smooth microbundles. We will use bold capital Roman letters for microbundle morphisms. For a representative of a microbundle morphism $\bm{F}$, we will typically use $F$ to denote the corresponding map between (subsets of) total spaces and $f$ to denote the map between base spaces. A microbundle morphism between smooth microbundles will be assumed to be smooth unless otherwise specified. Suppose $\exx$ is a topological microbundle over $B$ and $\bm{G}: \exx \to \exx'$ is a microbundle morphism. If $A\subset B$, we write $\bm{G}|_A$ to denote $\bm{G}|_{\exx|_A}$.

\subsection{Classifying Spaces for Microbundles and the Kirby-Siebenmann Classification of Manifold Structures}
\label{ss:KS background}

    We define $\TOPn{n}$ to be the topological group of basepoint-preserving homeomorphisms from $\mathbb{R}^n$ to itself with the compact-open topology. Its classifying space $\bTOPn{n}$ classifies topological $n$-microbundles in the following relative sense. 

\begin{theorem}
[{\cite[Essay~IV~Proposition~8.1]{ks}}]
    \label{thm: exist class spaces}
    There exists an $n$-dimensional topological microbundle $\gamtopn{n}$ over $\bTOPn{n}$ satisfying the following universality condition.
    \\ 
    Suppose we have the following:
    \begin{enumerate}
        \item A Euclidean neighborhood retract\footnote{A topological space $X$ is a Euclidean neighborhood retract (ENR) if there exists a closed embedding $i:X\to\mathbb{R}^N$ for some $N$ and an open neighborhood $U$ of $i(X)$ such that $i(X)$ is a retract of $U$. In particular, topological manifolds are ENRs (see \cite{enr}, Corolary 4.6).} $M$,
        \item An open neighborhood $U$ of a closed subset $C$ in $M$,
        \item A topological $n$-microbundle $\exx$ over $M$,
        \item A topological $n$-microbundle map $\bm{F}_0:\exx |_U\to \gamtopn{n}$.
    \end{enumerate}
    Then, there is a topological $n$-microbundle map $\bm{F}:\exx \to \gamtopn{n}$ that agrees with $\bm{F}_0$ near $C$.
\end{theorem}

Taking the colimit over $n$ yields the stable topological microbundle $\gamtop$ over the classifying space $\BTOP$, which stably classifies topological microbundles.

By Proposition \ref{prop: sm mb iso to vb}, smooth $n$-microbundles are classified by the universal vector bundle $\gamon{n}$ over the space $\bon{n}$ and are stably classified by the universal stable vector bundle $\gamo$ over $\bo$.\footnote{In \cite{ks}, Kirby and Siebenmann use $\BDIFF(n)$, defined as the Grassmannian of $n$-planes in $\mathbb{R}^{n+k}$, for a fixed, large $k$, and a smooth microbundle which is universal for base spaces of dimension $<k$. However, since smooth microbundles are equivalent to vector bundles, we use $\bon{n}$ and the universal vector bundle.}

 The forgetful map $\On{n}\to\TOPn{n}$ induces the following topological classifying map of $\gamon{n}$: 
    
    \begin{center}
        \begin{tikzcd}
            \gamon{n}\arrow{d} \arrow{r}{\bm{J}_{\On{n}}}& \gamtopn{n}\arrow{d}\\
            \bon{n} \arrow{r}{j_{\On{n}}} & \bTOPn{n}
        \end{tikzcd}
    \end{center}
     We also have the stable version:

 \begin{center}
        \begin{tikzcd}
            \gamo\arrow{d} \arrow{r}{\bm{J}_{\oo}}& \gamtop\arrow{d}\\
            \bo \arrow{r}{j_{\oo}} & \BTOP
        \end{tikzcd}
    \end{center}

     Let $M$ be a topological manifold and let $a:M\to\BTOP$ be a stable classifying map of $\micro{M}$. In high dimensions, the ability to endow $M$ with a smooth structure is equivalent to the existence of a lift of $a$ to $\bo$ through $j_{\oo}$.

\begin{theorem}
  [{\cite[Essay IV, Theorem 10.1]{ks}}]
    \label{thm:KS10.1 sm version}
    Let $M$ be an $n$-dimensional topological manifold and $U$ an open neighborhood of a closed subset $C\subset M$. Suppose we have the following:
    \begin{enumerate}
        \item A smooth structure $\Sigma$ on $U$,
        \item A smooth\footnote{Since $\micro{M}\big|_U$ is isomorphic to $\micro{U}$ by Remark \ref{rem:restrict micro iso}, $\Sigma$ determines a smooth microbundle structure on $\micro{M}\big|_U$.} $n$-microbundle map $\bm{A}_0:\micro{M}\big|_U\to \gamon{n}$ over a map $a_0:U\to\bon{n}$,
        \item A topological $n$-microbundle map $\bm{A}:\micro{M}\to \gamtopn{n}$ over a map $a:M\to\bTOPn{n}$ such that the composition $\bm{J}_{\On{n}}\bm{A}_0$ is equal to $\bm{A}$ near $C$.
    \end{enumerate}
    \noindent If $n\geq6$ or if $n=5$ with $\partial M \subset C$, then there is a bijective correspondence between concordance classes of smooth structures on $M$ extending $\Sigma$ and vertical homotopy classes of lifts (relative to $C$) of $a:M\to\bTOPn{n} \subset \BTOP$ to $\bo$ equal to $a_0:U\to\bon{n}\subset\bo$ near $C$. Let $\Gamma$ be a smooth structure on $M$ extending $\Sigma$. Under this bijection, the concordance class of $\Gamma$ is sent to the vertical homotopy class of lifts that classifies the smooth tangent microbundle of $M_\Gamma$.
\end{theorem}

\begin{theorem} [{\cite[Essay~V,~Theorem~5.5]{ks}}] 
\label{thm: ks V 5.5}
    The homotopy groups $\pi_i\LL \TOP/\oo\RR$ of the homotopy fiber of $j_{\oo}$ are isomorphic to $\pi_i\LL\kk\LL\mathbb{Z}/2,3\RR\RR$ for $i\leq 6$ and $\Theta_i$ for $i\geq 5$, where $\Theta_i$ is the group of h-cobordism classes of homotopy $n$-spheres under connected sum.  
\end{theorem}

\begin{remark}
\label{rem: KS only obst}
Suppose $n\leq 7$. If we have the data described in Theorem \ref{thm:KS10.1 sm version}, the only obstruction to having a lift relative to $C$ of $a:M\to\bTOPn{n} \subset \BTOP$ to $\bo$ is a class $\ks\LL M,C\RR \in H^4\LL M,C; \mathbb{Z}/2\RR$, called the Kirby-Siebenmann class. If $n=6$ or $7$ (or $5$ with $C\subset \partial M$), then, by Theorem \ref{thm:KS10.1 sm version}, $\ks\LL M,C\RR$ is the only obstruction to extending $\Sigma$ to a smooth structure on $M$.
\end{remark}


\section{Formally Smooth Results for Infinite $\bm{\pi_2}$}
\label{sec:FS}

\subsection{Formally Smooth Manifolds}
\label{ss: FS manifolds}

\begin{definition}
    \label{def: fs man}
    A \textbf{formally smooth (FS)} $\bm{n}$\textbf{-manifold} is a tuple $\LL X,\xi_X,\bm{G}_X, \bm{H}_{\partial X}\RR$, where
    
    \begin{itemize}
        \item $X$ is a topological $n$-manifold with (possibly empty) boundary $\partial X \xhookrightarrow{\iota}X$,
        
        \item $\xi_X$ is a smooth $n$-microbundle over $X$,
        
        \item $\bm{G}_X:\xi_X \rightarrow \micro{X}$ is a topological microbundle isomorphism over the identity on $X$,
        
        \item $\bm{H}_{\partial X}:\xi_{\partial X}\oplus\varepsilon \xrightarrow{\cong}\xi_X\big|_{\partial X}$ is a smooth microbundle isomorphism, where $\xi_{\partial X}$ is a smooth $(n-1)$-microbundle over $\partial X$ and $\varepsilon$ is the trivial smooth line microbundle bundle,
    \end{itemize}
    
\noindent subject to the following condition. Since $\micro{X}\big|_{\partial X}=\partial X\times\widehat{X}=\partial X \times \LL X \cup_{\partial X \times \{0\}} \partial X\times[0,1)\RR$, we have a natural decomposition \[\LL \micro{X}|_{\partial X}\RR \setminus\micro{\partial X} = \underbrace{\big( \partial X \times\LL X\setminus\partial X\RR \big)}_{=:\LL \micro{X}|_{\partial X}\RR^-} \sqcup \underbrace{\big( \partial X \times \LL \partial X \times (0,1)\RR\big)}_{=:\LL \micro{X}|_{\partial X}\RR^+}.\] Viewing $\xi_{\partial X}$ as a submicrobundle of $\xi_{\partial X}\oplus \varepsilon$ in the obvious way, we also have a natural decomposition \[\LL \xi_X|_{\partial X}\RR \setminus \bm{H}_{\partial X}\LL\xi_{\partial X}\RR = \underbrace{\bm{H}_{\partial X}\left(\xi_{\partial X} \oplus\varepsilon^-\right)}_{=:\LL \xi_X|_{\partial X}\RR^-}\sqcup \underbrace{\bm{H}_{\partial X}\left(\xi_{\partial X} \oplus\varepsilon^+\right)}_{=:\LL  \xi_X|_{\partial X}\RR^+},\] where $\varepsilon^-$ and $\varepsilon^+$ are the negative and positive components of $\varepsilon$, respectively. We require that $\bm{G}_{X}|_{\partial X}$ send $\LL \xi_X|_{\partial X}\RR^-$ to $\LL \micro{X}|_{\partial X}\RR^-$ and $\LL  \xi_X|_{\partial X}\RR^+$ to $\LL \micro{X}|_{\partial X}\RR^+$. We then define the boundary of $\LL X,\xi_X,\bm{G}_X, \bm{H}_{\partial X}\RR$ to be $\LL \partial X, \xi_{\partial X},\bm{G}_{\partial X}\RR$, where $\bm{G}_{\partial X}: \xi_{\partial X}\to \micro{\partial X}$ is given by $\bm{G}_X\bm{H}_{\partial X}\big|_{\xi_{\partial X}}$.

\end{definition}

\begin{remark}
\label{rem: smooth has canon FS}
    Let $X$ be a smooth manifold equipped with a collar neighborhood of its boundary $N\xrightarrow{\cong}\partial X\times I$. Then, $X$ has a canonical FS structure, where $\xi_X$ is the smooth tangent microbundle and $\bm{G}_X$ is the identity. Taking $\xi_{\partial X}$ to be the smooth tangent microbundle of $\partial X$, the collar neighborhood $N$ determines the smooth microbundle isomorphism $\bm{H}_{\partial X}$.
\end{remark}

We will typically suppress boundary data in our notation and write $(X,\xi_X,\bm{G}_X)$ for an FS manifold with or without boundary, unless it is unclear to do so.

\begin{definition}
    \label{def:fs morphism}
    An \textbf{FS morphism} from $\LL X_1,\xi_{X_1},\bm{G}_{X_1}\RR$ to $\LL X_2,\xi_{X_2},\bm{G}_{X_2}\RR$ is a smooth microbundle morphism $\bm{F}:\xi_{X_1}\to \xi_{X_2}$ over a continuous map $f:X_1 \to X_2$, such that the following diagram commutes.

\begin{center}
 \begin{tikzcd}

 \micro{X_1} \arrow{rrrrr}{\bm{D}f} \arrow[swap]{dddr} &&&&&\micro{X_2} \arrow{dddl}\\
 & \xi_{X_1} \arrow[swap]{ul}{\bm{G}_{X_1}} \arrow{dd} \arrow{rrr}{\bm{F}}&&& \xi_{X_2} \arrow{ur}{\bm{G}_{X_2}} \arrow{dd} \\
 \\
 &X_1 \arrow{rrr}{f} &&& X_2

 \end{tikzcd}
 \end{center}  

    \noindent Composition is well-defined via appropriate restrictions of the domains and codomains of the representative maps between total spaces. We note that an \textbf{FS isomorphism} is an FS morphism such that $\bm{F}$ is a smooth microbundle isomorphism over a homeomorphism $f$.
\end{definition}

\begin{proposition}
\label{prop: homeo is fs iso}
    Let $X_1$ and $X_2$ be homeomorphic topological manifolds. If $X_2$ has an FS structure, then $X_1$ has an FS structure such that they are FS isomorphic.

    \begin{proof}
            
    Let $\LL X_2, \xi_{X_2}, \bm{G}_{X_2}\RR$ be an FS structure on $X_2$ and let $f:X_1\to X_2$ be a homeomorphism. Define $\xi_{X_1}$ as the pullback smooth microbundle $f^*\xi_{X_2}$, and define $\bm{F}$ to be the induced smooth microbundle isomorphism over $f$. Then, there is an FS structure $\LL X_1,\xi_{X_1},\bm{G}_{X_1}\RR$, where $\bm{G}_{X_1}=\bm{D}f^{-1}\bm{G}_{X_2}\bm{F}$. By construction, $\bm{F}$ constitutes an FS isomorphism from $\LL X_1,\xi_{X_1},\bm{G}_{X_1}\RR$ to $\LL X_2,\xi_{X_2},\bm{G}_{X_2}\RR$.
    \end{proof}
\end{proposition}

\begin{proposition}
\label{prop: product of FS}
    The Cartesian product of two FS manifolds, with at most one having nonempty boundary, has a natural FS structure.
    \begin{proof}
      \sloppy  Let $\LL X_1,\xi_{X_1},\bm{G}_{X_1}\RR$ and $\LL X_2,\xi_{X_2},\bm{G}_{X_2}\RR$ be FS manifolds. Then, $X_1\times X_2$ has a natural FS structure $\LL X_1\times X_2, \xi_{X_1}\boxplus  \xi_{X_2},\bm{G}_{X_1}\boxplus \bm{G}_{X_2}\RR$. Suppose $\LL X_2,\xi_{X_2},\bm{G}_{X_2}\RR$ has boundary $\LL \partial X_2,\xi_{\partial X_2},\bm{G}_{\partial X_2}\RR$ with a smooth microbundle isomorphism $\bm{H}_{\partial X_2}: \xi_{\partial X_2}\oplus \varepsilon\to \xi_{X_2}|_{\partial X_2}$. Then, the boundary of  $\LL X_1\times X_2, \xi_{X_1}\boxplus \xi_{X_2},\bm{G}_{X_1}\boxplus \bm{G}_{X_2}\RR$ is $\LL X_1\times\partial X_2, \xi_{ X_1}\boxplus \xi_{\partial X_2},\bm{G}_{X_1}\boxplus \bm{G}_{\partial X_2}\RR$ with a smooth microbundle isomorphism $\bm{H}_{\partial \LL X_1\times X_2\RR}=\mathrm{id}_{\xi_{X_1}}\boxplus \bm{H}_{\partial X_2}: \xi_{X_1}\boxplus \LL \xi_{\partial X_2}\oplus \varepsilon\RR\to \xi_{X_1}\boxplus \xi_{X_2}|_{\partial X_2}$.
    \end{proof}
\end{proposition}

\begin{proposition}
\label{prop: triv ks is bts}
    A topological manifold of dimension $\leq 7$ can be endowed with an FS structure if and only if its Kirby-Siebenmann class vanishes. 
\end{proposition}

The proof follows from the results below of Kirby-Siebenmann and Freedman-Quinn.

\begin{theorem}[{\cite[Essay~V,~(4.6)~and~Theorem~5.3]{ks}}]
\label{thm: incl conn}
    \sloppy For $n\geq5$, the inclusion of homotopy fibers $\TOPn{n}/\On{n}\to \TOP/\oo$ is $(n+2)$-connected.
\end{theorem}

\begin{theorem}[{\cite[Theorem~8.7A]{freedmanquinn}}]
\label{thm: 4D incl conn}
    The inclusion $\TOPn{4}/\On{4}\to \TOP/\oo$ is $5$-connected.
\end{theorem}

 \begin{proof}[Proof of Proposition \ref{prop: triv ks is bts}]
        The forward direction follows immediately from Remark \ref{rem: KS only obst}. The backward direction in dimension $\leq3$ is trivial since all manifolds are smoothable. Let $X$ be a topological $n$-manifold with $4\leq n\leq 7$ and $\ks(X)=0$. Let $a_n:X\to\bTOPn{n}$ be a classifying map of $\micro{X}$ with stabilization $a:X\to\BTOP$. By Remark \ref{rem: KS only obst}, we have a lift $\widetilde{a}:X \to \bo$ of $a$. By Theorems \ref{thm: incl conn} and \ref{thm: 4D incl conn}, we may pull it back to a lift $\widetilde{a_n}:X\to\bon{n}$ of $a_n$, as is illustrated in the diagram below. The squares in the diagram commute and the triangles commute up to homotopy. 

\begin{center}
 \begin{tikzcd}
&&& \gamon{n} \arrow{rr} \arrow{dd} \arrow{dr} && \gamo \arrow{dr} \arrow{dd}\\
\micro{X} \arrow{dd} \arrow{rrrr} \arrow{rrru} &&&& \gamtopn{n} \arrow{dd} \arrow{rr} &&\gamtop \arrow{dd} \\
&&& \bon{n} \arrow{rr} \arrow{dr} && \bo \arrow{dr} \\
X \arrow[swap]{rrrr}{a_n} \arrow{rrru}{\widetilde{a_n}} &&&& \bTOPn{n} \arrow{rr} && \BTOP
\end{tikzcd}
 \end{center}  

\noindent Pulling $\gamon{n}$ back along $\widetilde{a_n}$ produces a smooth microbundle over $X$ that is topologically isomorphic to $\micro{X}$.
    \end{proof}

\sloppy For a connected $4$-manifold $M$, $H^4\LL M, \partial M;\mathbb{Z}/2\RR \cong \mathbb{Z}/2$, so we can view $\ks\LL M,\partial M\RR$ as an element of $\{0,1\}$, called the Kirby-Siebenmann invariant. The Kirby-Siebenmann invariant is additive under connected sum \cite[Theorem 9.2]{survey}. Thus, $\ks\LL \meight \RR = 2\ks \LL M^4_{E_8}\RR = 0$.\footnote{We can also see this as follows. Freedman's theorem tells us that, if the intersection form is even, the Kirby-Siebenmann invariant is equal to the signature divided by $8$. Since the signatures of $E_8\oplus E_8$ and $E_{16}$ are $16$, $\ks(\meight)=\ks(\msixt)=0$. This way, we also see that a single copy of $M^4_{E_8}$ has \emph{non}vanishing Kirby-Siebenmann invariant.} Since $\meight$ and $\msixt$ are stably homeomorphic, they differ by a connected summand that is smoothable and therefore has trivial Kirby-Siebenmann invariant. Thus, $\ks\LL\msixt\RR=\ks\LL\meight\RR=0$.

\begin{corollary}
\label{cor: meight and msixt are fs}
The topological manifolds $\meight$ and $\msixt$ have FS structures which are unique up to FS-isomorphism. 

\begin{proof}
     Existence follows immediately from Proposition \ref{prop: triv ks is bts}, and uniqueness follows from obstruction theorey since $H^3\LL\meight;\mathbb{Z}/2\RR=H^3\LL\msixt;\mathbb{Z}/2\RR=0$. There is therefore no ambiguity when considering $\meight$ and $\msixt$ as FS manifolds.
\end{proof}
\end{corollary}

\begin{remark}
    \label{rem: donaldson}
    Donaldson's Theorem \cite[Theorem~1]{donaldson1,donaldson2} states that, if the intersection form of a closed, oriented, connected, smooth $4$-manifold is positive-definite, it must be equivalent over the integers to the identity. The forms $E_8\oplus E_8$ and $E_{16}$ are not equivalent to the identity, so $\meight$ and $\msixt$ cannot be smoothed, despite having no bundle-theoretic smoothing obstructions.
\end{remark}

\begin{remark}
    \label{rem: unique iso class of fs on contractible}
    A contractible manifold has a unique isomorphism class of FS structures since all obstructions vanish and the relevant lifts are all homotopic. This highlights the extent to which the FS and smooth manifold categories are drastically different in $4$ dimensions. There are $4$-manifolds like $\meight$ and $\msixt$ with FS structures and no smooth structures, but also those with more isomorphism classes of smooth structures than of FS structures. For example, $\mathbb{R}^4$ has uncountably many non-diffeomorphic smooth structures, but only one isomorphism class of FS structures.
\end{remark}

    \subsection{Formally Smooth Bordisms and Gluing}
\label{ss: FS bord}
    
\begin{definition}
     Let $(X_1,\xi_{X_1},\bm{G}_{X_1})$ and $(X_2,\xi_{X_2},\bm{G}_{X_2})$ be closed FS manifolds. An \textbf{FS bordism} from $(X_1,\xi_{X_1},\bm{G}_{X_1})$ to $(X_2,\xi_{X_2},\bm{G}_{X_2})$ is an FS manifold $\LL M, \xi_M, \bm{G}_M\RR$ with boundary $\LL \partial M,\xi_{\partial M},\bm{G}_{\partial M}\RR$ together with the following data:

     \begin{itemize}
         \item A decomposition of the boundary: \[\LL \partial M,\xi_{\partial M},\bm{G}_{\partial M}\RR = \LL (\partial M)_1,\xi_{(\partial M)_1},\bm{G}_{(\partial M)_1}\RR \sqcup \LL (\partial M)_2,\xi_{(\partial M)_2},\bm{G}_{(\partial M)_2}\RR,\] 

         \item FS isomorphisms $\bm{F_i:}(X_i,\xi_{X_i},\bm{G}_{X_i})\to \LL (\partial M)_i,\xi_{(\partial M)_i},\bm{G}_{(\partial M)_i}\RR$, expressed as the following commutative diagram:
         \begin{center}
             \begin{tikzcd}
                  \micro{X_i} \arrow{rrrrr}{\bm{D}f_i} \arrow[swap]{dddr} &&&&&\micro{(\partial M)_i} \arrow{dddl}\\
 & \xi_{X_i} \arrow[swap]{ul}{\bm{G}_{X_i}} \arrow{dd} \arrow{rrr}{\bm{F}_i}&&& \xi_{(\partial M)_i} \arrow{ur}{\bm{G}_{(\partial M)_i}} \arrow{dd} \\
 \\
 &X_i \arrow{rrr}{f_i} &&& (\partial M)_i&,
             \end{tikzcd}
         \end{center}
         \item Smooth microbundle isomorphisms $\bm{L}_i:\xi_{X_i}\oplus\varepsilon\to\xi_{(\partial M)_i}\oplus\varepsilon$ over $f_i$ subject to the following conditions:
         \begin{enumerate}
             \item $\bm{L}_i$ restricts to $\bm{F}_i$ on $\xi_{X_i}$.
             \item $\bm{L}_1$ sends $\varepsilon^+$ to $\varepsilon^+$ and $\bm{L}_2$ sends $\varepsilon^+$ to $\varepsilon^-$. (This is the analogue of choosing an inward/outward pointing normal vector for each boundary component in the smooth case.)
         \end{enumerate} 
     \end{itemize}
\noindent We will often not work directly with the actual boundary and will notationally conflate the source and target manifolds with the corresponding parts of the boundary. We will therefore use the following conventions: Let $\iota:\partial M\to M$ and $\bm{H}_{\partial M}:\xi_{\partial M}\oplus \varepsilon\to \xi_M|_{\partial M}$ be the inclusion and smooth microbundle isomorphism as in Definition \ref{def: fs man}. Let "$\bm{H}_{(\partial M)_i}$" denote the restriction of $\bm{H}_{\partial M}$ to $\xi_{(\partial M)_i}\oplus \varepsilon$. We use the following shorthand notation:
\begin{itemize}
\doublespacing
    \item "$\iota_i:X_i\to M$" refers to the composition $X_i\xrightarrow{f_i}(\partial M)_i\xrightarrow{\iota}M$,
    \item "$\xi_M\big|_{X_i}$" or "$\iota_i^*\xi_M$" refers to $f_i^*\xi_M\big|_{(\partial M)_i}$,
    \item "$\bm{H}_{X_i}: \xi_{X_i}\oplus \varepsilon\to \iota_i^*\xi_M$" refers to $\xi_{X_i}\oplus\varepsilon\xrightarrow{\bm{L}_i}\xi_{(\partial M)_i}\oplus\varepsilon \xrightarrow{\bm{H}_{(\partial M)_i}}\xi_M|_{(\partial M)_i}$.
\end{itemize}

\end{definition} 

\begin{remark}
    Let $\LL X,\xi_X,\bm{G}_X\RR$ be a closed FS manifold. By Proposition \ref{prop: product of FS}, there is a canonical FS structure on $X\times I$, which ensures the existence of identity bordisms.
\end{remark}

           Let $\LL \varnothing\to M_1\xleftarrow{\iota_1}X , \xi_{M_1},\bm{G}_{M_1},\LL \bm{H}_{X}\RR _1\RR$ and $\LL X \xrightarrow{\iota_2} M_2 \leftarrow \varnothing, \xi_{M_2},\bm{G}_{M_2},\LL \bm{H}_{X}\RR _2 \RR$ be FS bordisms from $\varnothing$ to $\LL X, \xi_X,\bm{G}_X\RR$ and $\LL X, \xi_X,\bm{G}_X\RR$ to $\varnothing$, respectively. Note that the index set $\{1,2\}$ is now being used to label the two bordisms being glued rather than differentiating the in- versus out-boundary of a single bordism. The microbundle data can be glued so that the result is an FS bordism with underlying topological bordism $M_1\cup_X {M_2}$.
       We glue the smooth microbundles $\xi_{M_1}$ and $\xi_{{M_2}}$ as follows: Let $\LL W,W'_1,\LL H_X\RR_1\RR$ and $\LL W,W'_2,\LL H_X\RR_2\RR$ be representatives of $\LL \bm{H}_{X}\RR _1$ and $\LL \bm{H}_{X}\RR _2$ such that $\LL H_X\RR_1$ and $\LL H_X\RR_2$ are homeomorphisms. This provides an identification $\xi_{M_1}\big|_X\to\xi_{M_2}\big|_X$, which allows us to glue $\xi_{M_1}$ and $\xi_{M_2}$. We will refer to the resulting topological space as $\xi_{M_1}\Cup_X\xi_{M_2}$.

\begin{lemma}
    \label{vb gluing}
    The space $\xi_{M_1}\Cup_X\xi_{M_2}$ is a smooth microbundle over $M_1\cup_XM_2$.

    \begin{proof}
        For all points of $M_1\cup_X M_2$ away from the gluing locus, we may import the local trivializations from $\xi_{M_1}$ or $\xi_{M_2}$. 
        For a point $x\in X$, let $V_1\subset \xi_{M_1}$ and $V_2\subset \xi_{M_2}$ be trivial open neighborhoods of $\iota_1(x)$ and $\iota_2(x)$, respectively, such that $\LL V_1\cap \xi_{M_1}|_{X}\RR\subset W_1'$ and $\LL V_2\cap \xi_{M_2}|_{X}\RR\subset W_2'$. By "$\subset \xi_{M_i}$" and "$\cap \hspace{1mm}\xi_{M_i}|_{X}$", we mean subset of and intersection with the total space of the appropriate microbundle. Under the identification $\xi_{M_1}\big|_{X}\to\xi_{M_2}\big|_{X}$, $V_1$ and $V_2$ may not form a trivial open neighborhood of $x$, but we may choose one $V\subset V_1\Cup_X V_2$. Then, the local trivialization data from $\xi_{M_1}$ and $\xi_{M_2}$ glue into a local trivialization of $\xi_{M_1}\Cup_X\xi_{M_2}$ at $x$. Therefore, $\xi_{M_1}\Cup_X\xi_{M_2}$ is a topological microbundle. The fiberwise smoothness condition for smooth microbundles laid out in definition \ref{def: smooth micro} is met automatically since the gluing is done via a smooth microbundle isomorphism.
            \end{proof}

\end{lemma}

We can glue the tangent microbundles $\micro{M_1}$ and $\micro{M_2}$ in a way that allows for gluing the isomorphisms $\bm{G}_{M_1}:\xi_{M_1}\rightarrow\micro{M_1}$ and $\bm{G}_{M_2}:\xi_{M_2}\rightarrow\micro{M_2}$. Let $\LL Y_1,Y_1',G_{M_1}\RR$ and $\LL Y_2,Y_2',G_{M_2}\RR$ be representatives of $\bm{G}_{M_1}$ and $\bm{G}_{M_2}$ such that $G_{M_1}$ and $G_{M_2}$ are homeomorphisms and such that $\LL Y_1\cap \xi_M|_{X_1}\RR\subset W_1'$ and $\LL Y_2\cap \xi_M|_{X_1}\RR\subset W_2'$. Let $Y\subset \xi_{M_1}\Cup_X\xi_{M_2}$ be an open neighborhood of the zero section such that $\LL Y\cap \xi_{M_1} \RR\subset Y_1$ and $\LL Y\cap \xi_{M_2} \RR\subset Y_2$. Let $U=Y\cap\left(\xi_{M_1}\Cup_X\xi_{M_2}\right)\big|_X$. This gives an identification $h:G_{M_1}|_X\left(U\right)\to G_{M_2}|_X\left(U\right)$ with which we can glue $\micro{M_1}$ and $\micro{M_2}$ as illustrated in Figure 5.1. We again denote the glued space by $\micro{M_1}\Cup_X\micro{M_2}$.

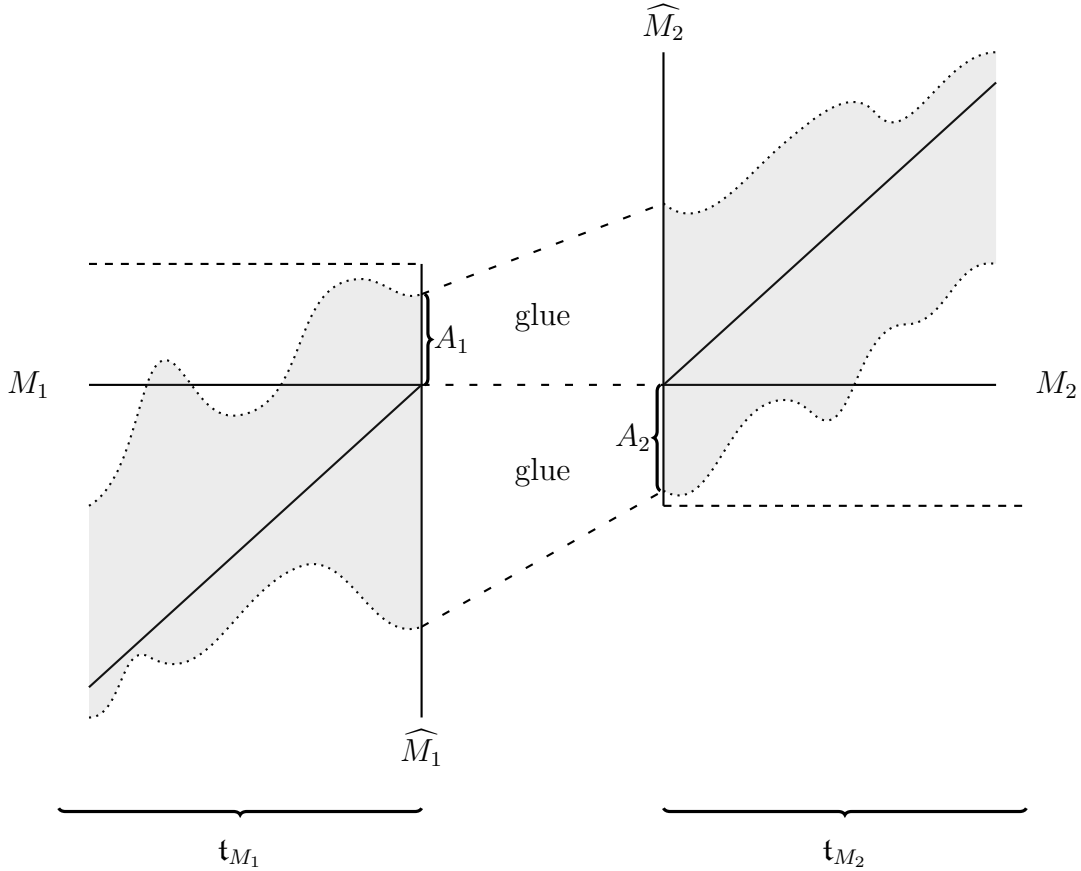
\begin{figure}[H]
    \centering
    
    \label{fig:gluing diagram}
   
    \begin{tikzpicture}[scale=0.8]
  \draw[thick] (-7.5,0) -- (-2,0);
  \draw[thick][dashed] (-7.5,2) -- (-2,2);
  \draw[thick] (-2,-5.5) -- (-2,2);
  \draw[thick] (-7.5,-5) -- (-2,0);
  \draw[name path=lower left][thick][dotted] (-7.5,-2) to[in=130, out = 30] (-6,0.25) to[in=190,out =-50] (-5,-0.5) to[out=10, in=180] (-3,1.75) to[out=0, in = 200] (-2,1.5);
  \draw[name path=upper left][thick][dotted] (-7.5,-5.5) to[out=0, in=150] (-6.5,-4.5) to[out=-30, in=200] (-4,-3) to[out=20, in=200] (-2,-4);
  \draw[thick] (2,0) -- (7.5,0);
  \draw[thick][dashed] (2,-2) -- (8,-2);
  \draw[thick] (2,-2) -- (2,5.5);
  \draw[thick] (2,0) -- (7.5,5);
  \draw[name path=lower right][thick][dotted] (2,-1.75) to[out=-30, in=180] (4,-0.25) to[out=0, in=140] (4.5,-0.5) to[out=-40, in=180] (6,1) to[out=0, in=170] (7.5,2);
  \draw[thick][dotted][name path=upper right] (2,3) to[out=-40, in=220] (4,4) to[out=40, in=130] (5.5,4.5) to[out=-50, in=180] (7.5,5.5);

  \draw[thick][loosely dashed] (-2,1.5) -- (2,3);
  \draw[thick][loosely dashed] (-2,0) -- (2,0);
  \draw[thick][loosely dashed] (-2,-4) -- (2,-1.75);
  \tikzfillbetween[of=upper left and lower left]{gray, opacity=0.15};
  \tikzfillbetween[of=upper right and lower right]{gray, opacity=0.15};

  \node at (0,1.125) {$\mathrm{glue}$};
    \node at (0,-1.4375) {$\mathrm{glue}$};
 \node at (-2,-6) {$\widehat{M_1}$};
  \node at (-8.5,0) {$M_1$};
   \node at (2,6) {$\widehat{M_2}$};
  \node at (8.5,0) {$M_2$};
     \draw[very thick,decorate,decoration=brace] (-1.95,1.5) -- (-1.95,0);
       \node at (-1.5,0.75) {$A_1$};

     \draw[very thick,decorate,decoration=brace] (1.95,-1.75) -- (1.95,0);
       \node at (1.5,-0.875) {$A_2$};

      \draw[very thick, decorate, decoration = brace] (-2,-7) -- (-8,-7);
      \node at (-5,-7.75) {$\micro{M_1}$};
       \draw[very thick, decorate, decoration = brace] (8,-7) -- (2,-7);
      \node at (5,-7.75) {$\micro{M_2}$};
\end{tikzpicture}
\caption{Schematic of gluing tangent microbundles}
\end{figure}

\begin{lemma}
\label{lem:gluing micros}
 The space $\micro{M_1}\Cup_X\micro{M_2}$ is a topological microbundle over $M_1\cup_XM_2$ and is isomorphic to the tangent microbundle $\micro{M_1\cup_XM_2}$.
\begin{proof}
It is clear that we have a section and projection map given by gluing the diagonal maps and projections onto the first factor. It now suffices to show that an open neighborhood of the section is homeomorphic to an open neighborhood of the diagonal in $\micro{M_1\cup_XM_2}$ in a way that preserves fibers. Let $A_1=G_{M_1}|_X\left(U\right)\cap\left(\micro{M_1}|{_X}\right)^+$ and $A_2=G_{M_2}|_X\left(U\right)\cap\left(\micro{M_2}|{_X}\right)^+$. The homeomorphism $h$ gives fiber-preserving identifications of $A_1$ with an open neighborhood of the diagonal in $X\times M_2$ and of $A_2$ with an open neighborhood of the diagonal in $X\times M_1$. Let $N_1\xrightarrow{\cong}X\times I$ and $N_2\xrightarrow{\cong} X\times I$ be collar neighborhoods of $X$ in $M_1$ and $M_2$. We may then extend $h$ along these neighborhoods to get fiber-preserving identifications of $A_1\times I\subset \left(\micro{M_1}|{_X}\right)^+$ with an open neighborhood of the diagonal in $N_1\times M_2$ and of $A_2\times I\subset\left(\micro{M_2}|{_X}\right)^+$ with an open neighborhood of the diagonal in $N_2\times M_1$. Let $E$ be an open neighborhood of the section in $\micro{M_1}\Cup_X\micro{M_2}$ such that $\LL E\cap \left(\micro{M_1}|{_X}\right)^+ \RR \subset A_1\times I$ and $\LL E\cap \left(\micro{M_2}|{_X}\right)^+\RR \subset A_2\times I$. Then, the extensions of $h$ glue, along with $\mathrm{id}_{M_1\times M_1}$ and $\mathrm{id}_{M_2\times M_2}$ into a fiber-preserving homeomorphism $J$ from $E$ to an open neighborhood of the diagonal in $\micro{M_1\cup_XM_2}$. We can then pull back local trivializations from $\micro{M_1\cup_XM_2}$ to $\micro{M_1}\Cup_X\micro{M_2}$. Thus, $\micro{M_1}\Cup_X\micro{M_2}$ is indeed a topological microbundle over $M_1\cup_XM_2$ and $J$ determines a topological microbundle isomorphism $\bm{J}$ over the identity.
\end{proof}

\end{lemma}
This proves the following proposition.
\begin{proposition}
    \label{prop: gluing nullbord}
    The following data constitute an FS bordism:
    \begin{enumerate}
        \item The topological bordism $\left(\varnothing\rightarrow M_1\cup_X M_2\leftarrow \varnothing \right)$,
        \item The smooth microbundle $\xi_{M_1}\Cup_X\xi_{M_2}$,
        \item The topological microbundle isomorphism $\bm{J}\circ\left(\bm{G}_{M_1}\Cup_X\bm{G}_{M_2}\right):\xi_{M_1}\Cup_X\xi_{M_2}\to \micro{M_1\cup_XM_2}$. \qed
    \end{enumerate}
\end{proposition}

\begin{remark}
    The gluing argument extends immediately to the case where the in-boundary of $M_1$ and the out-boundary of $M_2$ are nonempty. Although there were choices involved in defining $\xi_{M_1}\Cup_X\xi_{M_2}$, $\micro{M_1}\Cup_X\micro{M_2}$, and $\bm{J}\circ\left(\bm{G}_{M_1}\Cup_X\bm{G}_{M_2}\right)$, any other choices will result in a bordism of the same FS isomorphism class relative to the boundary.
\end{remark}

We thus have the following corollary.

\begin{corollary}
    \label{cor: gluing fs}
    Gluing of FS bordisms is well-defined up to FS isomorphism. \qed
\end{corollary}

\begin{proposition}
\label{prop: conn sum of fs}
    The connected sum of connected FS manifolds is a well-defined operation up to FS isomorphism.
    \begin{proof}

       Let $M$ and $N$ be connected topological $n$-manifolds with FS structures. For simplicity, assume they have empty boundaries. Let $M'$ and $N'$ be the result of removing the interior of an embedded $n$-disk from $M$ and $N$, respectively. A choice of collar neighborhoods determines FS boundary data for $\partial M'\subset M'$ and $\partial N'\subset N'$. By Remark \ref{rem: unique iso class of fs on contractible}, there is a unique isomorphism class of FS structures on a disk, so we may view $M'$ and $N'$ as FS bordisms to and from, respectively, $S^{n-1}$ with the canonical FS structure. We may then glue as described above.
    \end{proof}
\end{proposition}

\subsection{Formally Smooth TQFTs}
\label{ss:FS TQFTs}

\begin{definition}
\label{def: FS bord cat}
We define the \textbf{FS Bordism Category}, denoted $\fscat{n}$, as follows: 
    
\end{definition}
\noindent 

The objects are closed, $(n-1)$-dimensional FS manifolds, and the morphisms are FS bordisms between them, modulo orientation-preserving FS isomorphisms relative to the boundary. In other words, the FS bordisms $\big( X_1 \xrightarrow{\iota_1} M \xleftarrow{\iota_2} X_2, \xi_M, \bm{G}_M, \bm{H}_{\LL X_1\sqcup X_2\RR} \big)$ and $\big( X_1 \xrightarrow{\iota'_1} M' \xleftarrow{\iota'_2} X_2, \xi_{M'}, \bm{G}_{M'},\bm{H}'_{\LL X_1\sqcup X_2\RR}\big)$ from $(X_1,\xi_{X_1},\bm{G}_{X_1})$ to $(X_2,\xi_{X_2},\bm{G}_{X_2})$ are equivalent if there exists an orientation-preserving homeomorphism $f: M \to M'$ and a smooth microbundle isomorphism $\bm{F}:\xi_{M} \to \xi_{M'}$ such that the following diagram commutes, where $\bm{H}_{X_i}$ and $\bm{H}'_{X_i}$ denote the restrictions of $\bm{H}_{\LL X_1\sqcup X_2\RR}$ and $\bm{H}'_{\LL X_1\sqcup X_2\RR}$ to $X_i$.

 \begin{center}
 \begin{tikzcd}
 \\
 \micro{M} \arrow{rrrrrr}{\bm{D}f} \arrow[swap]{ddddddr}&&&&&&\micro{M'} \arrow{ddddddl}\\
 &\xi_{M{}} \arrow{ddddd} \arrow{ul}[near start]{\bm{G}_{M}} \arrow{rrrr}{\bm{F}}&&&&\xi_{M'} \arrow[swap,near start]{ur}{\bm{G}_{M'}} \arrow[swap]{ddddd}\\
 \\&&\iota_i^*\xi_{M} \arrow{uul}   \arrow[swap]{ddr} & \arrow[swap]{l}{\bm{H}_{X_i}}\xi_{X_i}\oplus \varepsilon \arrow{dd} \arrow{r}{\bm{H}'_{X_i}} &\iota_i'^*\xi_{M'} \arrow{uur}  \arrow{ddl}
 \\
 \\&&&X_i \arrow[swap]{dll}{\iota_i} \arrow{drr}{\iota'_i}
 \\
 &M \arrow{rrrr}{f}&&&& M' &.
 \\
\end{tikzcd}
 \end{center}  

\begin{remark}
   \sloppy If $\LL X_1, \xi_{X_1},\bm{G}_{X_1}\RR$ and $\LL X_2, \xi_{X_2},\bm{G}_{X_2}\RR$ are closed FS $(n-1)$-manifolds which are FS isomorphic, then they are also isomorphic in $\fscat{n}$. An FS isomorphism between $\LL X_1, \xi_{X_1},\bm{G}_{X_1}\RR$ and $\LL X_2, \xi_{X_2},\bm{G}_{X_2}\RR$ allows us to view the identity bordism on $\LL X_1, \xi_{X_1},\bm{G}_{X_1}\RR$ as a bordism from $\LL X_1, \xi_{X_1},\bm{G}_{X_1}\RR$ to $\LL X_2, \xi_{X_2},\bm{G}_{X_2}\RR$.
\end{remark}

\begin{remark}
   By Proposition \ref{prop: conn sum of fs}, we have a well-defined notion of stable isomorphism of FS manifolds, and can thus extend Proposition \ref{prop: cant dist stab iso} to the FS category.  By Proposition \ref{prop: homeo is fs iso} and Corollary \ref{cor: meight and msixt are fs}, $\meight$ and $\msixt$ are stably isomorphic in the FS category, and thus cannot be distinguished by ordinary $4$-dimensional FS TQFTs.
\end{remark}

\begin{theorem}
   \label{thm: main thm FS}

Ordinary $(4+m)$-dimensional FS TQFTs fail to distinguish $\meight\times S^m$ and $\msixt\times S^m$.

    \begin{proof}
        The proof is analogous to that of Theorem \ref{thm: main thm top}. Let $\mathcal{Z}:\fscat{4+m}\to \vectc$ be an ordinary FS TQFT. By Remark \ref{rem: smooth has canon FS} and Proposition \ref{prop: product of FS}, the dimensional reduction functor $\fscat{4}\xrightarrow{-\times S^m}\fscat{4+m}$ is well-defined. By Proposition \ref{prop: cant dist stab iso}, the dimensional reduction of $\mathcal{Z}$ along $S^m$ cannot distinguish $\meight$ and $\msixt$, so $\mathcal{Z}$ cannot distinguish $\ssm{\meight}$ and $\ssm{\msixt}$.
        \end{proof}
\end{theorem}


\section{Smooth Results for Infinite $\bm{\pi_2}$}
\label{sec:smooth}

We will use "$X_\Sigma$" to denote a smooth manifold with underlying topological manifold $X$ and smooth structure $\Sigma$, and we will denote the smooth tangent microbundle of $X_\Sigma$ by $\tau_{X_\Sigma}$. Recall from Remark \ref{rem: smooth has canon FS} that, given a collar neighborhood of $\partial X$, there is a canonical FS structure on $X_\Sigma$ given by $\LL X,\tau_{X_\Sigma},\bm{\mathrm{Id}}_{X\times \widehat{X}}\RR$.

\begin{theorem}
\label{thm: bord equiv}
    For $n\geq6$, the categories $\smcat{n}$ and $\fscat{n}$ are equivalent.
   \end{theorem} 
The proof will be based on the following lemma. Fix a closed smooth $n$-manifold $X_\Sigma$. Let $\mathcal{M}^{\Sm}_{X_\Sigma}$ denote the category whose objects are compact smooth manifolds with boundary $X_\Sigma$ equipped with a collar neighborhood $N$ of $X_\Sigma$ and whose morphisms are diffeomorphisms relative to $X_\Sigma$. We will denote an object of $\mathcal{M}^{\Sm}_{X_\Sigma}$ by $\LL M_\Gamma, N\RR$. Let $\mathcal{M}^{\FS}_{X_\Sigma}$ be the category whose objects are FS manifolds with boundary $\LL X, \tau_{X_\Sigma},\bm{\mathrm{Id}}_{X\times X}\RR$ and whose morphisms are FS isomorphisms relative to the boundary.

\begin{lemma}
    There is a natural functor $\mathscr{H}_{X_\Sigma}: \mathcal{M}^{\Sm}_{X_\Sigma}\to \mathcal{M}^{\FS}_{X_\Sigma}$ that is essentially bijective for $n\geq 5$.
    \begin{proof}
        We define $\mathscr{H}_{X_\Sigma}$ as follows:

        Objects: $\mathscr{H}_{X_\Sigma}\LL M_\Gamma,N\RR=\LL M,\tau_{M_\Gamma},\bm{\mathrm{Id}}_{M\times \widehat{M}}\RR$,

        Morphisms: $\mathscr{H}_{X_\Sigma}\LL f\RR=\bm{D}f$, illustrated in the diagram below.

        \begin{center}
   \begin{tikzcd}
  \micro{M} 
    \arrow[rrrrr, "{\bm{D}f}"] 
    \arrow[swap, dddr] 
  &&&&& 
  \micro{M'} 
    \arrow[dddl] \\
  & \tau_{M_\Gamma} 
      \arrow[swap, ul, "{\bm{\mathrm{Id}}_{M\times \widehat{M}}}" {inner sep=0.5pt, yshift=-0.3ex}]
      \arrow[dd] 
      \arrow[rrr, "{\bm{D}f}"] 
    &&& 
      \tau_{M'_{\Gamma'}} 
        \arrow[ur, "{\bm{\mathrm{Id}}_{M'\times \widehat{M'}}}" {inner sep=0.5pt, yshift=-0.3ex}] 
        \arrow[dd] \\
  \\
  & M \arrow[rrr, "f"] &&& M'
\end{tikzcd}
\end{center}

\sloppy \emph{Essentially injective:} Let $\LL M_\Gamma,N\RR$ and $\LL M'_{\Gamma'},N'\RR$ be objects such that $\mathscr{H}_{X_\Sigma}\LL M_\Gamma,N\RR\cong\mathscr{H}_{X_\Sigma}\LL M'_{\Gamma'},N'\RR$. We then have a homeomorphism (relative to $X$) $f:M\to M'$ and a smooth microbundle isomorphism $\bm{F}:\tau_{M_\Gamma}\to\tau_{M'_{\Gamma'}}$ that agrees topologically with the micro-differential.
      Let $f^*\Gamma'$ be the smooth structure on $M$ with tangent bundle $f^*\tau_{M'_{\Gamma'}}$ that arises from pulling $\Gamma'$ back along $f$. Since $\Gamma$ and $f^*\Gamma'$ agree on $X_\Sigma$, we may assume that they agree on some open collar neighborhood $U_{\widehat{\Sigma}}$ of $X_\Sigma$, where $\widehat{\Sigma}$ is the product smooth structure.\footnote{A perturbation that yields agreement on $U$ will not change the diffeomorphism type relative to $X$.} Fix some smooth\footnote{Recall that $\bon{n}$ is the colimit of finite Grassmannians, which are smooth manifolds. By "smooth map", we mean one that is the colimit of smooth maps.} classifying map $a_0:N\to \bo$. Any two classifying maps $a,a':M\to\bo$ that classify $\tau_{M_\Gamma}$ and $f^*\tau_{M'_{\Gamma'}}$, respectively, (and extend $a_0$) are homotopic since the microbundles are isomorphic. Furthermore, the fact that $\bm{F}$ agrees with the topological micro-differential implies that the classifying maps are vertically homotopic over $\BTOP$. By Theorem \ref{thm:KS10.1 sm version}, $\Gamma$ and $f^*\Gamma'$ are concordant smooth structures on $M$ relative to $U_{\widehat{\Sigma}}$. By the concordance implies isotopy theorem (\cite[Essay~I,~Theorem~4.1]{ks}), they are in fact diffeomorphic relative to $U_{\widehat{\Sigma}}$ (and, in particular, relative to $X_\Sigma$).

       \emph{Essentially surjective}: Let $\LL M, \xi_M, \bm{G}_M\RR$ be an object of $\mathcal{M}^{\FS}_{X_\Sigma}$, i.e. an FS manifold with boundary $\LL X,\tau_{X_{\Sigma}},\bm{\mathrm{Id}}_{X\times X}\RR$. Let $N_{\widehat{\Sigma}}$ be as above. Fix some smooth (with respect to $\widehat{\Sigma}$) map $\bm{A}_0:\tau_{N_{\widehat{\Sigma}}} \to \gamo$ over a map $a_0:N\to\bo$ and a map $\bm{A}:\xi_M\to\gamo$ that agrees with $\bm{A}_0$ near $N$ (under the identification of $\tau_{N_{\widehat{\Sigma}}} $ with $\xi_M\big|_N$). Let $a:M\to\bo$ be the map covered by $\bm{A}$. Then, $k: M \xrightarrow{a} \bo \xrightarrow{j_{\oo}}\BTOP$ stably classifies $\micro{M}$. By construction, $a:M\to \bo$ is a lift of $k$ that agrees with $a_0$ near $N$, and thus determines a concordance class of smooth structures on $M$ extending $\widehat{\Sigma}$, by Theorem \ref{thm:KS10.1 sm version}. Let $\Gamma$ be one such smooth structure, and let $a':M\to \bo$ be a classifying map of $\tau_{M_\Gamma}$ that agrees with $a_0$ near $N$. Again by Theorem \ref{thm:KS10.1 sm version}, $a$ and $a'$ are vertically homotopic over $\BTOP$, so $\mathscr{H}_{X_\Sigma}\LL M_\Gamma,N\RR = \LL M, \tau_{M_\Gamma}, \bm{\mathrm{Id}}_{M\times \widehat{M}}\RR$ is FS isomorphic to $\LL M, \xi_M, \bm{G}_M\RR$ over the identity. In other words, there is a smooth microbundle isomorphism $\bm{F}$ such that the following diagram commutes.

               \begin{center}
   \begin{tikzcd}
  \micro{M} 
    \arrow[rrrrr, "{\bm{D}(\mathrm{Id}_M)}"] 
    \arrow[swap, dddr] 
  &&&&& 
  \micro{M} 
    \arrow[dddl] \\
  & \tau_{M_\Gamma} 
      \arrow[swap, ul, "{\bm{\mathrm{Id}}_{M\times \widehat{M}}}" {inner sep=0.5pt, yshift=-0.3ex}]
      \arrow[dd] 
      \arrow[rrr, "{\bm{F}}"] 
    &&& 
      \xi_M 
        \arrow[ur, "{\bm{G}_M}" {inner sep=0.5pt, yshift=-0.3ex}] 
        \arrow[dd] \\
  \\
  & M \arrow[rrr, "\mathrm{Id}_M"] &&& M
\end{tikzcd}
\end{center} \end{proof}
\end{lemma}

Although collar neighborhood data is necessary in order to determine natural FS data on a smooth manifold with boundary $X_\Sigma$, the choice of collar has no effect on the corresponding isomorphism class in either groupoid $\mathcal{M}^{\Sm}_{X_\Sigma}$ or $\mathcal{M}^{\FS}_{X_\Sigma}$. We can therefore drop collar notation when discussing the induced map on isomorphism classes $\pi_0\mathscr{H}_{X_\Sigma}: \pi_0\mathcal{M}^{\Sm}_{X_\Sigma}\to \pi_0\mathcal{M}^{\FS}_{X_\Sigma}$. If $[M_\Gamma]$ is a morphism in $\smcat{n}$ from $X_{1,\Sigma_1}$ to $X_{2,\Sigma_2}$, then $\pi_0\mathscr{H}_{\partial M_{\partial \Gamma}}$ determines a morphism in $\fscat{n}$ from $\mathscr{H}_\varnothing \LL X_{1,\Sigma_1}\RR = \LL X_1, \tau_{X_{1,\Sigma_1}}, \bm{\mathrm{Id}}_{X_1\times X_1}\RR$ to $\mathscr{H}_\varnothing \LL X_{2,\Sigma_2}\RR=\LL X_2, \tau_{X_{2,\Sigma_2}}, \bm{\mathrm{Id}}_{X_2\times X_2}\RR$ represented by $\LL M, \tau_{M_\Gamma}, \bm{\mathrm{Id}}_{M\times \widehat{M}}\RR$.

\begin{lemma}
    The map $\mathscr{F}_n:\smcat{n}\to\fscat{n}$ defined below is a functor.

Objects: For a closed, smooth $(n-1)$-manifold $X_\Sigma$, $\mathscr{F}_n\LL X_\Sigma\RR=\mathscr{H}_{\varnothing}\LL X_\Sigma\RR$.

   Morphisms: For a morphism $[M_\Gamma]$ in $\smcat{n}$ from $X_{1,\Sigma_1}$ to $X_{2,\Sigma_2}$, $\mathscr{F}_n\LL [M_\Gamma]\RR=\left[ \LL M, \tau_{M_\Gamma}, \bm{\mathrm{Id}}_{M\times \widehat{M}}\RR\right]$. 
  \begin{proof}

      \emph{Identities}: Since $\tau_{(X\times I)_{\widehat{\Sigma}}}$ is isomorphic to the external Whitney sum $\tau_{X_\Sigma}\boxplus\varepsilon$, $\mathscr{H}_{X_\Sigma}\big((X\times I)_{\widehat{\Sigma}}\big)$ is FS isomorphic (relative to the boundary) to the identity FS bordism on $\LL X,\tau_{X_\Sigma},\bm{\mathrm{Id}}_{X\times X}\RR$.

 \emph{Composition}: Let $\left[M_{1,\Gamma_1}\right]$ and $\left[M_{2,\Gamma_2}\right]$ be smooth morphisms from $\varnothing$ to $X_\Sigma$ and from $X_\Sigma$ to $\varnothing$, respectively, with composition given by $\left[\LL M_1 \cup_X M_2\RR_\Gamma\right]$ for some appropriate smooth structure $\Gamma$. The image under $\mathscr{F}_{n}$ of
$\left[\LL M_1 \cup_X M_2\RR_\Gamma\right]$ is
$\left[
 \LL  M_1 \cup_X M_2 ,\,
  \tau_{\LL M_1 \cup_X M_2 \RR_\Gamma},\,
  \bm{\mathrm{Id}}_{\LL M_1\cup_X M_2\RR\times\LL M_1\cup_X M_2\RR}
\RR\right]$. Since $\tau_{\LL M_1 \cup_X M_2 \RR_\Gamma}$ and $\tau_{M_{1,\Gamma_1}}\Cup_X \tau_{M_{2,\Gamma_2}}$ are isomorphic, this image is equal to $\left[\LL M_1\cup_XM_2,\tau_{M_{1,\Gamma_1}}\Cup_X \tau_{M_{2,\Gamma_2}},\bm{\mathrm{Id}}_{\LL M_1\times \widehat{M_1}\RR}\Cup_X \bm{\mathrm{Id}}_{\LL M_2\times \widehat{M_2}\RR}\RR\right]$. This argument extends immediately to the case of non-empty boundary components not in the gluing locus.
  \end{proof}
    
\end{lemma}

    \begin{proof}[Proof of Theorem \ref{thm: bord equiv}]

For $n\geq6$, $\mathscr{F}_n$ is essentially surjective and fully faithful, and is thus part of an equivalence of categories. Essential surjectivity follows from the essential surjectivity of $\mathscr{H}_{\varnothing}$ in dimension $\geq 5$. Fullness and faithfulness follow from the essential surjectivity and essential injectivity, respectively, of $\mathscr{H}_{X_{1,\Sigma_1}\sqcup X_{2,\Sigma_2}}$ in dimensions $\geq 6$.
\end{proof}

We will use $\mathscr{F}^{-1}_n$ to denote a weak inverse of $\mathscr{F}_n$.

\begin{theorem}
    \label{thm: sm cant dist htpy}
    For $m\geq 2$, $(4+m)$-dimensional ordinary smooth TQFTs fail to distinguish $\ssm{\meight}$ and $\ssm{\msixt}$.
    \begin{proof}
        Let $\mathcal{Z}:\smcat{4+m}\to \vectc$ be an ordinary smooth TQFT. Since $\mathscr{F}_n^{-1}$ is clearly symmetric monoidal, we may define the ordinary $(4+m)$-dimensional FS TQFT $\mathcal{Y}$ to be the following composition:
    
        \begin{center}
    $\mathcal{Y}:\fscat{4+m}\xlongrightarrow{\mathscr{F}^{-1}_{4+m}} \smcat{4+m}\xlongrightarrow{\mathcal{Z}} \vectc$.
\end{center}
\noindent By Theorem \ref{thm: main thm FS}, $\mathcal{Y}$ cannot distinguish the FS manifolds $\meight\times S^m$ and $\msixt\times S^m$. So, $\mathcal{Z}$ cannot distinguish the smooth manifolds $\meight\times S^m$ and $\msixt\times S^m$, (which we define in the smooth category to be the images of their FS counterparts under $\mathscr{F}^{-1}_n$).
    \end{proof}
\end{theorem}


\section{Special Consideration in Dimension 5}
\label{sec: dim 5}

In \cite[Theorem~3]{kt}, Kreck and Teichner proved that ordinary TQFTs (in particular, a class called positive TQFTs) can distinguish simply connected, closed, smooth $5$-manifolds. In this section, we provide a pair of closed, non-homotopy equivalent, smooth, \emph{non}-simply connected $5$-manifolds that cannot be distinguished by ordinary $5$-dimensional TQFTs. The topological and FS versions of this result are simply Theorems \ref{thm: main thm top} and \ref{thm: main thm FS} in the case $m=1$.

For the smooth case, we are not in a high enough dimension for the functor defined in the previous section $\mathscr{F}_5:\smcat{5}\to\fscat{5}$ to be an equivalence. It is, however, fully faithful. We use this fact to construct a functor $\mathscr{G}: \fscat{4}\xlongrightarrow{-\times S^1}\smcat{5}$.

We first assign a smooth structure to every closed topological $3$-manifold in the following way. For every closed connected topological $3$-manifold, we may arbitrarily choose a fixed smooth structure. Then, we assign to every closed $3$-manifold the disjoint union of the smooth structures assigned to its connected components. We now construct the functor $\mathscr{G}$ as follows.

Objects: Let $\LL X, \xi_X,\bm{G}_X\RR$ be a closed FS $3$-manifold, and let $\Sigma$ be the smooth structure on $X$ chosen as described above. Let $\widehat{\Sigma}$ be the product smooth structure on $X\times S^1$ given by $\Sigma$ and the standard smooth structure on $S^1$. We then define $\mathscr{G}\LL X,\xi_X,\bm{G}_X\RR = \LL X\times S^1\RR_{\widehat{\Sigma}}$.

\sloppy Morphisms: Let $\LL M,\xi_M,\bm{G}_M\RR$ be a $4$-dimensional FS bordism from  $\LL X_1, \xi_{X_1},\bm{G}_{X_1}\RR$ to $\LL X_2,\xi_{X_2},\bm{G}_{X_2}\RR$. By obstruction theory, $3$-manifolds have a unique isomorphism class of FS structures, so we may also consider $\LL M,\xi_M,\bm{G}_M\RR$ to be an FS bordism from $\LL X_1,\tau_{X_{1,\Sigma_1}},\bm{\mathrm{Id}}_{X_1\times X_1}\RR$ to $\LL X_2,\tau_{X_{2,\Sigma_2}},\bm{\mathrm{Id}}_{X_2\times X_2}\RR$, which yields a $5$-dimensional FS bordism $\LL M\times S^1,\xi_M\boxplus\tau_{S^1}, \bm{G}_M \boxplus \bm{\mathrm{Id}}_{S^1\times S^1}\RR$ from $\LL X_1\times S^1,\tau_{{(X_1\times S^1)}_{\widehat{\Sigma_1}}}, \bm{\mathrm{Id}}_{\LL X_1\times S^1\RR \times \LL X_1\times S^1\RR}\RR$ to $\LL X_2\times S^1, \tau_{{(X_2\times S^1)}_{\widehat{\Sigma_2}}}, \bm{\mathrm{Id}}_{\LL X_2\times S^1\RR \times \LL X_2\times S^1\RR}\RR$. Since $\mathscr{F}_5$ is fully faithful, this bordism determines a smooth morphism $\left[(M\times S^1)_{\Gamma}\right]$ in $\smcat{5}$ from $(X_1\times S^1)_{\widehat{\Sigma_1}}$ to $(X_2\times S^1)_{\widehat{\Sigma_2}}$, which we define to be $\mathscr{G} \left[\LL M,\xi_M,\bm{G}_M\RR\right]$.

\begin{theorem}
\label{thm: 5 dim smooth}
   Ordinary $5$-dimensional smooth TQFTs fail to distinguish $\meight\times S^1$ and $\msixt\times S^1$.
    \begin{proof}
        Let $\mathcal{Z}:\smcat{5}\to\vectc$ be an ordinary $5$-dimensional smooth TQFT. Consider the following composition:

        \begin{center}
            $\mathcal{W}:\fscat{4}\xlongrightarrow{\mathscr{G}}\smcat{5}\xlongrightarrow{\mathcal{Z}}\vectc$
        \end{center}

        \noindent The functor $\mathscr{G}$ is symmetric monoidal due to our method of assigning smooth structures to objects of $\fscat{4}$. Thus, $\mathcal{W}$ is an ordinary TQFT and therefore cannot distinguish $\meight$ from $\msixt$, so $\mathcal{Z}$ cannot distinguish $\meight\times S^1$ from $\msixt\times S^1$.
    \end{proof}
\end{theorem}


\section{Non-Spinnable Counter-Examples}
\label{sec: non spin}

In this section, we first argue that $\meight\times S^m$ and $\msixt\times S^m$ are spinnable manifolds. Second, we show that $\LL \meight \# \mathbb{C}P^2 \RR \times S^m$ and $\LL \msixt \# \mathbb{C}P^2 \RR \times S^m$ are non-homotopy equivalent, non-spinnable manifolds which cannot be distinguished by ordinary TQFTs.

The standard definition of a spin structure on a smooth $n$-manifold $M$ is a lift $M\to \BSpin(n)$ of a map $M\to\bon{n}$ classifying the tangent bundle. One can extend this definition to topological manifolds using the tangent microbundle and an analogous space $\BTopSpin(n)$. We will, however, adopt a purely cohomological definition of spinnability using Wu classes.\footnote{Stiefel-Whitney classes were defined in the topological setting by Fadell in \cite{fadell}. He observed (as did Thom \cite{thom}) that the proof of the equivalence between the vanishing of the first two Stiefel-Whitney classes and the existence of a spin structure on a smooth manifold does not depend on the existence of a smooth structure. We could also then define spinnable topological manifolds via vanishing Stiefel-Whitney classes rather than Wu classes. (See \cite[Chapter~7]{survey} for a concise overview.)}

\begin{definition}
    The $\bm{i}$\textbf{th Wu class} of an $n$-manifold $M$ is the unique class $v_i(M)\in H^i(M;\mathbb{Z}/2)$ such that, for any $\alpha\in H^{n-i}(M;\mathbb{Z}/2)$, $v_i(M)\smile \alpha=\sqi{i}(\alpha)$, where $\sqi{i}$ denotes the $i$th Steenrod square. The \textbf{total Wu class} is the sum \[v(M)=\displaystyle\sum_{i=0}^\infty v_i(M).\]
\end{definition}

If $M$ is a smooth manifold, the Wu formula \cite{wu} relates the total Wu class to the total Stiefel-Whitney class as follows:

\begin{center}
    $\sq(v(M))=w(M)$,
\end{center}

\noindent where $\sq$ and $w(M)$ are the total Steenrod square and total Stiefel-Whitney class, respectively. From this, we may deduce that $w_1(M)=v_1(M)$ and $w_2(M)=\LL v_1(M)\RR^2+v_2(M))$. Thus, the vanishing of the first two Stiefel-Whitney classes is equivalent to the vanishing of the first two Wu classes. Since the former is equivalent to the existence of a spin structure on a smooth manifold, we use the following definition.

\begin{definition}
    \label{def: top spin}
    A topological manifold $M$ is \textbf{spinnable} if $v_1(M)$ and $v_2(M)$ vanish.
\end{definition}

\begin{remark}
\label{rem: spin if simp con and even int}
    It follows from the properties of Steenrod squares that a closed, simply-connected $4$-manifold is spinnable if and only if its intersection form is even. This is explained in \cite[Proposition~8.4~and~Remark~8.9]{survey}.
    \end{remark}
    
\begin{lemma}
\label{lem: prod of spin iff}
    Let $M$ and $N$ be connected topological manifolds. If $N$ is spinnable, then $M\times N$ is spinnable if and only if $M$ is spin.
    \begin{proof}
        This statement follows quickly from basic properties of Steenrod squares. Let $p_1$ and $p_2$ be the projections out of $M\times N$ onto the first and second factor. Then, a class in $H^*\LL M\times N\RR$ is a sum of classes of the form $p_1^*(\alpha)\smile p_2^*(\beta)$. By the Cartan formula and naturality of Steenrod squares, we have \[\sqi{k}\big( p_1^*(\alpha)\smile p_2^*(\beta)\big)=\displaystyle\sum_{i+j=k}p_1^*\big(\sqi{i}(\alpha)\big)\smile p^*_2\big(\sqi{j}(\beta)\big).\] Since $N$ is spinnable, $\sqi{1}$ and $\sqi{2}$ act trivially on $\beta$, so, for $i=1,2$, \[\sqi{i}\big( p_1^*(\alpha)\smile p_2^*(\beta)\big) = p_1^*\big(\sqi{i}(\alpha)\big)\smile p_2^*\LL \sqi{0}(\beta)\RR = p_1^*\big(\sqi{i}(\alpha)\big)\smile p^*_2(\beta). \]
       If $M$ is spinnable, then it is clear that $\sqi{1}$ and $\sqi{2}$ act trivially on the product. If $M$ is not spinnable, then there is some class $\alpha\in H^*(M,\mathbb{Z}/2)$ such that either $\sqi{1}(\alpha)$ or $\sqi{2}(\alpha)$ is nonzero. Taking $\beta$ to be $1\in H^0\LL N;\mathbb{Z}/2\RR$, it follows from the injectivity of $p^*_1$ that either $\sqi{1}$ or $\sqi{2}$ acts non-trivially on the product.
    \end{proof}
\end{lemma}

\begin{corollary}
\label{cor: are spin}
    The manifolds $\meight\times S^m$ and $\msixt\times S^m$ are spinnable. \qed
\end{corollary}

\begin{corollary}
\label{cor: are not spin}
    The manifolds $\LL \meight \# \mathbb{C}P^2 \RR \times S^m$ and $\LL \msixt \# \mathbb{C}P^2 \RR \times S^m$ are not spinnable.
    \begin{proof}
        Since the intersection form of $\mathbb{C}P^2$ is $\langle1\rangle$, $\meight \# \mathbb{C}P^2$ and $\msixt \# \mathbb{C}P^2$ are not spinnable by Remark \ref{rem: spin if simp con and even int}. By Lemma \ref{lem: prod of spin iff}, the products with $S^m$ are also not spinnable.
    \end{proof}
\end{corollary}

\begin{proposition}
\label{prop: non spin not htpy eq}
    The manifolds $\LL \meight \# \mathbb{C}P^2 \RR \times S^m$ and $\LL \msixt \# \mathbb{C}P^2 \RR \times S^m$ are not homotopy equivalent and, if $m$ is even and $\geq2$, they are still not homotopy equivalent after stabilization. They are nevertheless indistinguishable by ordinary $(4+m)$-dimensional TQFTs in the topological, FS, and smooth categories.
    \begin{proof}
        We show using Magma (see \ref{appendix}ppendix) that the bilinear forms $E_8\oplus E_8\oplus \langle 1 \rangle$ and $E_{16}\oplus \langle 1 \rangle$ are not isometric. The first statement then follows from Lemmas \ref{lem: diff int not htpy eq after cross} and \ref{lem: diff int not stab htp}. 

        Recall that $\meight$ and $\msixt$ are stably isomorphic. Then, $\meight \# \mathbb{C}P^2$ and $\msixt \#\mathbb{C}P^2$ are stably isomorphic and therefore cannot be distinguished by $4$-dimensional ordinary TQFTs, by Proposition \ref{prop: cant dist stab iso}. We may then apply our dimensional reduction arguments to show that $\LL \meight \# \mathbb{C}P^2 \RR \times S^m$ and $\LL \msixt \# \mathbb{C}P^2 \RR \times S^m$ are indistinguishable by ordinary TQFTs.
    \end{proof}
\end{proposition}


\section{Highly-Connected Counter-Examples}
\label{sec: high conn}

For all $n\geq2$, there are pairs of stably-diffeomorphic, $(2n-1)$-connected, closed, smooth $4n$-manifolds $M^{4n}$ and $N^{4n}$ with non-isometric intersection forms which one can construct by plumbing (see \cite{kt}\cite{wall}\cite{kreckandschafer}). We may choose to have $M^{4n}$ and $N^{4n}$ arise from plumbing according to the $E_8\oplus E_8$ and $E_{16}$ lattices, respectively, giving us direct analogues of our $4$-dimensional "building blocks" used in previous sections. The following proposition is an analogue of Proposition \ref{prop:88 16 product not htpy eq or stab iso} and follows immediately from Lemmas \ref{lem: diff int not htpy eq after cross} and \ref{lem: diff int not stab htp}

\begin{proposition}
    \label{prop: high conn 88 16 product not htpy eq or stab iso}
    The manifolds $\ssm{M^{4n}}$ and $\ssm{N^{4n}}$ are not homotopy equivalent. If $m$ is even and nonzero, they are also not homotopy equivalent after stabilization. \qed
\end{proposition}

\begin{proposition}
    \label{prop: high conn indist}
    In the smooth, FS, and topological categories, for $n\geq2$, $(4n+m)$-dimensional ordinary TQFTs cannot distinguish $\ssm{M^{4n}}$ and $\ssm{N^{4n}}$.
    \begin{proof}
        In all three cases, the proof is identical. Let $\mathcal{Z}$ be an ordinary, $(4n+m)$-dimensional TQFT, and let $\mathcal{W}$ be its reduction along $S^m$: \[\mathcal{W}:\bord_{4n}\xrightarrow{-\times S^m}\bord_{4n+m}\xrightarrow{\mathcal{Z}}\vectc.\] Since $M^{4n}$ and $N^{4n}$ are stably isomorphic, they are indistinguishable to $\mathcal{W}$ by Proposition \ref{prop: cant dist stab iso}. Thus, $\mathcal{Z}$ cannot distinguish $\ssm{M^{4n}}$ and $\ssm{N^{4n}}$. 
    \end{proof}
\end{proposition}

Restricting to taking products with spheres of dimensions high enough to maintain the connectivity of $M^{4n}$ and $N^{4n}$ yields the following variant of Proposition \ref{prop: high conn indist}.

\begin{proposition}
\label{prop: high conn main theorem}
    For all $n\geq2$ and $m\geq 2n$, $\ssm{M^{4n}}$ and $\ssm{N^{4n}}$ are $(2n-1)$-connected and are indistinguishable by ordinary $(4n+m)$-dimensional TQFTs. \qed
\end{proposition}

\begin{remark}
    For $n\geq 2$, $M^{4n}$ and $N^{4n}$ are at least $2$-connected, and are thus spinnable. By Lemma \ref{lem: prod of spin iff}, $\ssm{M^{4n}}$ and $\ssm{N^{4n}}$ are spinnable as well.
\end{remark}


\appendix
\section{Appendix -- Magma Computation}
\label{appendix}

The following is Magma code that shows that $E_8\oplus E_8\oplus \langle 1 \rangle$ and $E_{16}\oplus \langle 1 \rangle$ are not isometric. First, we define a candidate for a basis of the $E_{16}$ lattice, by mimicking the well-known basis for $E_8$.

\begin{verbatim}
> M := Matrix(RationalField(), 16, 16, 
[2, 0, 0, 0, 0, 0, 0, 0, 0, 0, 0, 0, 0, 0, 0, 0,
-1, 1, 0, 0, 0, 0, 0, 0, 0, 0, 0, 0, 0, 0, 0, 0, 
0, -1, 1, 0, 0, 0, 0, 0, 0, 0, 0, 0, 0, 0, 0, 0, 
0, 0, -1, 1, 0, 0, 0, 0, 0, 0, 0, 0, 0, 0, 0, 0, 
0, 0, 0, -1, 1, 0, 0, 0, 0, 0, 0, 0, 0, 0, 0, 0, 
0, 0, 0, 0, -1, 1, 0, 0, 0, 0, 0, 0, 0, 0, 0, 0, 
0, 0, 0, 0, 0, -1, 1, 0, 0, 0, 0, 0, 0, 0, 0, 0, 
0, 0, 0, 0, 0, 0, -1, 1, 0, 0, 0, 0, 0, 0, 0, 0, 
0, 0, 0, 0, 0, 0, 0, -1, 1, 0, 0, 0, 0, 0, 0, 0, 
0, 0, 0, 0, 0, 0, 0, 0, -1, 1, 0, 0, 0, 0, 0, 0, 
0, 0, 0, 0, 0, 0, 0, 0, 0, -1, 1, 0, 0, 0, 0, 0, 
0, 0, 0, 0, 0, 0, 0, 0, 0, 0, -1, 1, 0, 0, 0, 0, 
0, 0, 0, 0, 0, 0, 0, 0, 0, 0, 0, -1, 1, 0, 0, 0,
0, 0, 0, 0, 0, 0, 0, 0, 0, 0, 0, 0, -1, 1, 0, 0, 
0, 0, 0, 0, 0, 0, 0, 0, 0, 0, 0, 0, 0, -1, 1, 0, 
½, ½, ½, ½, ½, ½, ½, ½, ½, ½, ½, ½, ½, ½, ½, ½]);
\end{verbatim}

\noindent Next, we check that it does indeed define the $E_{16}$ lattice.

Input:
\begin{verbatim}
    >LatticeWithBasis(M);
\end{verbatim}

Output: 
\begin{verbatim}
    Lattice of rank 16 and degree 16
Determinant: 1
Basis:
( 4  0  0  0  0  0  0  0  0  0  0  0  0  0  0  0)
(-2  2  0  0  0  0  0  0  0  0  0  0  0  0  0  0)
( 0 -2  2  0  0  0  0  0  0  0  0  0  0  0  0  0)
( 0  0 -2  2  0  0  0  0  0  0  0  0  0  0  0  0)
( 0  0  0 -2  2  0  0  0  0  0  0  0  0  0  0  0)
( 0  0  0  0 -2  2  0  0  0  0  0  0  0  0  0  0)
( 0  0  0  0  0 -2  2  0  0  0  0  0  0  0  0  0)
( 0  0  0  0  0  0 -2  2  0  0  0  0  0  0  0  0)
( 0  0  0  0  0  0  0 -2  2  0  0  0  0  0  0  0)
( 0  0  0  0  0  0  0  0 -2  2  0  0  0  0  0  0)
( 0  0  0  0  0  0  0  0  0 -2  2  0  0  0  0  0)
( 0  0  0  0  0  0  0  0  0  0 -2  2  0  0  0  0)
( 0  0  0  0  0  0  0  0  0  0  0 -2  2  0  0  0)
( 0  0  0  0  0  0  0  0  0  0  0  0 -2  2  0  0)
( 0  0  0  0  0  0  0  0  0  0  0  0  0 -2  2  0)
( 1  1  1  1  1  1  1  1  1  1  1  1  1  1  1  1)
Basis Denominator: 2
\end{verbatim}

The above confirms that our matrix defines a full rank, unimodular, degree 16 matrix. We now check parity:

Input: \begin{verbatim}
    >IsEven(LatticeWithBasis(M));
\end{verbatim}

Output: 
\begin{verbatim}
    true
\end{verbatim}

\noindent Next, we verify that it is distinct from the $E_8\oplus E_8$ lattice.

Input:
\begin{verbatim}
>IsIsometric(LatticeWithBasis(M),DirectSum(Lattice("E",8),Lattice("E",8)));
\end{verbatim}

Output: 
\begin{verbatim}
    false
\end{verbatim}

We have now shown that the lattice with basis $M$ is indeed $E_{16}$. 

\noindent Next, we check if the two lattices remain distinct after direct sum with $\langle 1 \rangle$.

Input: \begin{verbatim}
    > Z := Matrix(IntegerRing(), 1, 1, [1]);
    >IsIsometric(DirectSum(LatticeWithBasis(M),LatticeWithBasis(Z)),
     DirectSum(DirectSum(Lattice("E",8),Lattice("E",8)),LatticeWithBasis(Z)));
\end{verbatim}

Output: \begin{verbatim}
    false
\end{verbatim}

\printbibliography
\end{document}